\documentclass[11pt,a4paper,oneside,openany]{amsart}
\usepackage{amsfonts,amsmath,amscd,amssymb,amsbsy,amsthm,amstext,amsopn,amsxtra,mathtools}
\usepackage{color,mathrsfs,verbatim,subfigure}
\usepackage{extarrows}
\usepackage[all]{xy}
\usepackage{cleveref}
\usepackage{tikz}
\usepackage{tikz-cd}
\usepackage[enableskew, vcentermath]{youngtab}
\usetikzlibrary{shapes.geometric, calc,matrix,arrows,decorations.pathmorphing,arrows.meta}
\usepackage{caption}
\usepackage{enumitem}
\usepackage{pdfpages}
\usepackage[toc,page]{appendix}
\usepackage[T1]{fontenc}
\usepackage{textcomp}
\usepackage{amsmath}
\usepackage{stmaryrd,amssymb}
\usepackage{amscd}
\usepackage[utf8x]{inputenc}
\usepackage{t1enc}
\usepackage[mathscr]{eucal}
\usepackage{indentfirst}

\DeclareRobustCommand{\VAN}[3]{#2} % set up for citation

\theoremstyle{plain}
\newtheorem{thm}{Theorem}[section]
\newtheorem{lemma}[thm]{Lemma}
\newtheorem{cor}[thm]{Corollary}
\newtheorem{prop}[thm]{Proposition}
\newtheorem*{thmno}{Theorem}
\newtheorem*{corno}{Corollary}

\theoremstyle{definition}
\newtheorem{defn}[thm]{Definition}

\newtheorem{rmk}[thm]{Remark}

\newtheorem{example}[thm]{Example}
\newtheorem*{abusenotation}{Abuse of notation}

\renewcommand{\Re}{\operatorname{Re}}
\renewcommand{\Im}{\operatorname{Im}}

\newcommand{\CC}{\mathbb{C}}
\newcommand{\PP}{\mathbb{P}}
\newcommand{\QQ}{\mathbb{Q}}
\newcommand{\RR}{\mathbb{R}}
\newcommand{\ZZ}{\mathbb{Z}}

\newcommand{\cC}{\mathcal{C}}
\newcommand{\cE}{\mathcal{E}}
\newcommand{\cF}{\mathcal{F}}
\newcommand{\cH}{\mathcal{H}}

\newcommand{\cJ}{\mathcal{J}}

\newcommand{\cL}{\mathcal{L}}
\newcommand{\cM}{\mathcal{M}}
\newcommand{\cO}{\mathcal{O}}

\newcommand{\cS}{\mathcal{S}}

\newcommand{\cU}{\mathcal{U}}
\newcommand{\cV}{\mathcal{V}}
\newcommand{\cX}{\mathcal{X}}
\newcommand{\cY}{\mathcal{Y}}

\newcommand{\fM}{\mathfrak{M}}

\newcommand{\Alb}{\operatorname{Alb}}

\newcommand{\Aut}{\operatorname{Aut}}

\newcommand{\Bir}{\operatorname{Bir}}

\newcommand{\ch}{\operatorname{ch}}
\newcommand{\CH}{\operatorname{CH}}
\newcommand{\D}{\operatorname{D}}

\newcommand{\id}{\operatorname{id}}
\newcommand{\oH}{\operatorname{H}}

\newcommand{\Hom}{\operatorname{Hom}}

\newcommand{\Kum}{\operatorname{Kum}}
\newcommand{\Mon}{\operatorname{Mon}}

\newcommand{\Or}{\operatorname{O}}

\newcommand{\Ort}{\operatorname{\widetilde{O}}}

\newcommand{\Pic}{\operatorname{Pic}}

\newcommand{\Sym}{\operatorname{Sym}}
\newcommand{\td}{\operatorname{td}}

\date{}

\title[Monodromy of IHS manifolds of OG10 type]{On the monodromy group of desingularised moduli spaces of sheaves on K3 surfaces}
\author{Claudio Onorati}
\address{Department of Mathematics, University of Oslo, PO Box 1053, Blindern, 0316 Oslo, Norway -- claudion@math.uio.no}
\curraddr{Dipartimento di Matematica, Universit\`a degli studi di Roma Tor Vergata, via della Ricerca Scientifica 1, 00133 Roma, Italy}
\email{onorati@mat.uniroma2.it}

\begin{document}

\maketitle

%%%%%%%%%%%%%%%%%%%%%%%%%%%%%%%%%%%%%%%%%
\begin{abstract}
In this paper we prove a conjecture of Markman about the shape of the monodromy group of irreducible holomorphic symplectic manifolds of OG10 type. As a corollary, we also compute the locally trivial monodromy group of the underlying singular symplectic variety.
\end{abstract}

\tableofcontents

%%%%%%%%%%%%%%%%%%%%%%%%%%%%%%%%%%%%%%%%%
%%%%%%%%%%%%%%%%%%%%%%%%%%%%%%%%%%%%%%%%%
\section*{Introduction}
An irreducible holomorphic symplectic manifold is a compact K\"ahler manifold that is
simply connected with a unique closed non-degenerate holomorphic $2$-form. They are fundamental factors in the Beauville--Bogomolov
decomposition of compact K\"ahler manifolds with trivial canonical bundle. The first
example is in dimension $2$, where they are all $K3$ surfaces. It is very difficult to construct 
examples in higher dimension, and so far only four deformation types are known: two
families in any even dimension, namely the $K3^{[n]}$ and $\Kum^n$ types;
and two sporadic families in dimension $6$ and $10$, namely the OG6 and OG10 types.
The latter is the main object of this paper.

Like $K3$ surfaces, irreducible holomorphic symplectic manifolds are studied via their second 
integral cohomology group. More precisely, Beauville, Bogomolov and Fujiki independently noticed
that the group $\oH^2(X,\ZZ)$, where $X$ is an irreducible holomorphic
symplectic manifold, has a natural non-degenerate symmetric bilinear form, which generalises the
intersection product on a $K3$ surface (e.g.\ \cite{Beauville:c1=0}). With this bilinear form,
$\oH^2(X,\ZZ)$ becomes a lattice of signature $(3,b_2(X)-3)$. %Despite the case of $K3$ surfaces, it is not unimodular in general, nor it is known to be necessarily even, although it is in fact even in all the known examples. 
In all the examples $\oH^2(X,\ZZ)$ is an even lattice, but this is not known to hold in general.
Moreover, it is unimodular when $X$ is a $K3$ surface, but in other cases this is not necessarily true: in particular it is not true in all the known higher dimensional examples.
It is natural to study and understand the group $\Or(\oH^2(X,\ZZ))$ of isometries of the lattice, in particular those isometries which are geometrically meaningful (in a sense to be made precise in Definition~\ref{defn:parallel transport operators}). The subgroup of these geometrically meaningful isometries is denoted $\Mon^2(X)$, and its elements are called \emph{monodromy} operators: among the elements in $\Mon^2(X)$ there are for examples isometries of the form $f^*$, where $f\in\Aut(X)$ is an automorphism and, more generally, $f\in\Bir(X)$ is a birational automorphism (since the canonical bundle of an irreducible holomorphic symplectic manifold is nef, the pullback of a birational transformation is an isomorphism of $\oH^2(X,\ZZ)$ into itself, which preserves the lattice structure - see \cite{Huybrechts:BasicResults}). All isometries in $\Mon^2(X)$ preserve a natural orientation of the lattice (see Remark~\ref{rmk:orient pub}), so they lie in the subgroup $\Or^+(\oH^2(X,\ZZ))$ of orientation preserving isometries. There are isometries that do not preserve the orientation, like $-\id\in\Or(\oH^2(X,\ZZ))$, and that therefore are not monodromy operator. This can be explained from a moduli theory point of view as follows: if $\fM$ denotes the moduli space of marked irreducible holomorphic symplectic manifolds of fixed dimension and deformation type, then the pairs $(X,\eta)$ and $(X,-\eta)$ always belong to different connected components. As remarked before, birational automorphisms induces monodromy operators. On the other hand, any monodromy operator preserving the Hodge structure comes from a birational automorphism, up to some exceptional reflections associated with divisorial contractions (\cite[Theorem~6.18]{Markman:Survey}). We study some exceptional reflections in the context of manifolds of type OG10 in Section~\ref{section:exceptional reflections}.

The knowledge of the monodromy group is of paramount importance to study any aspect of
the geometry of irreducible holomorphic symplectic manifolds. It has been explicitly computed by 
Markman for manifolds of $K3^{[n]}$ type (see \cite{Markman:Monodromy1}), and
by Markman and Mongardi for manifolds of $\Kum^n$ type 
(see \cite{Markman:MonodromyKum} and \cite{Mongardi:Monodromy}). In both cases its
exact shape depends on the dimension, but the general feature is that 
$\Mon^2(K3^{[n]})=\Or^+(\oH^2(K3^{[n]},\ZZ))$ (when $n-1$ is a prime power),
while $\Mon^2(\Kum^n)\subset\Or^+(\oH^2(\Kum^2,\ZZ))$ has index $2$
(when $n+1$ is a prime power). The last fact is deeply related to the geometry: it
reflects the fact that there exist two families of manifolds of $\Kum^n$ type that are
generically non-birational, but Hodge isometric. 
This phenomenon was noticed by Namikawa in \cite{Namikawa:Counterexample} as a
counter-example to the birational Torelli theorem.
Finally, the monodromy group of manifolds of OG6 type has been recently computed by Mongardi and Rapagnetta (\cite{MongardiRapagnetta:MonodromyOG6}), who showed that it is maximal, i.e.\ $\Mon^2(\operatorname{OG6})=\Or^+(\oH^2(\operatorname{OG6},\ZZ))$.

In this paper, we address the remaining question of determining the monodromy group of manifolds of OG10 type. It was conjectured by Markman that 
$\Mon^2(\operatorname{OG10})=\Or^+(\oH^2(\operatorname{OG10},\ZZ))$
(see \cite[Conjecture~10.7]{Markman:Survey}). In \cite[Theorem~3.3]{Mongardi:Monodromy} Mongardi constructs an orientation preserving isometry that is not of monodromy type: unfortunately the construction is based on the work \cite{WandelMongardi:InducedAutomorphisms} that contains a mistake (see~\cite{Mongardi:MonodromyErratum}). %\cite{MongardiOnorati:BirGeomOG10}

Our main result is the following affirmative answer to Markman's conjecture.
\begin{thmno}[Theorem~\ref{thm:Mon^2}]
Let $X$ be an irreducible holomorphic symplectic manifold of OG10 type. Then
$\Mon^2(X)=\Or^+(\oH^2(X,\ZZ))$.
\end{thmno}
We give an explicit description of $\Mon^2(X)$ in terms of generators when $X=\widetilde{M}_S$ is the O'Grady
moduli space, namely the symplectic desingularisation of the moduli space $M_S$ of rank 
$2$ sheaves on a projective $K3$ surface $S$, with trivial determinant and
second Chern class of degree $4$ (see Example~\ref{example:O'Grady moduli space}).

As a straightforward corollary of this result (see \cite[Theorem~1.3]{Markman:Survey}), we get a strong version of the global Torelli theorem.
\begin{corno}[Global Torelli Theorem]
Let $X$ and $Y$ be two irreducible holomorphic symplectic manifolds of OG10 type. Then $X$ and $Y$ are bimeromorphic if and only if they are Hodge isometric.
\end{corno}
Here we say that $X$ and $Y$ are Hodge isometric if there exists an isometry between $\oH^2(X,\ZZ)$ and $\oH^2(Y,\ZZ)$ that respects the Hodge decomposition.

Let us outline the proof of Theorem~\ref{thm:Mon^2}.
The first step consists in producing monodromy operators. Partial results in this direction were
obtained by Markman in \cite{Markman:ModularGaloisCovers}, where he proved that the group
generated by two particular exceptional reflections is contained in the monodromy group. 
He worked in the family of O'Grady moduli spaces. Working in the same family, we study how the monodromy
group of the underlying $K3$ surface lifts to the monodromy group of the O'Grady moduli space
(see Theorem~\ref{thm:my operators O'Grady}). This result was expected, but,
to the best of the author's knowledge, there is no proof in the literature.

Using non-trivial results in 
birational geometry, like for example the termination of the minimal model program for irreducible 
holomorphic symplectic manifolds (\cite{LehnPacienza:MMP}), and the birational geometry 
of singular moduli spaces of sheaves on $K3$ surfaces 
(\cite{MeachanZhang:BirationalGeometry}), we study in Section~\ref{section:exceptional reflections} monodromy operators arising as exceptional reflections around divisors that are pullback of prime divisors of square $-2$ on the underlying singular moduli space.

More monodromy operators are constructed using the family of compactified intermediate Jacobian fibrations constructed by Laza, Sacc\`a and Voisin in \cite{LSV:O'Gr10}. 
If $V$ is a generic cubic fourfold, the LSV compactification of the fibration whose fibres are 
intermediate Jacobians of smooth linear sections of $V$ is an irreducible holomorphic symplectic
manifold of OG10 type. Working in this family, we study how the monodromy of the
cubic fourfold lifts to the monodromy of the OG10 manifold. 
An explicit parallel transport operator between this family and the family of O'Grady moduli spaces 
is constructed in Section~\ref{section:bridge} (see Theorem~\ref{thm:LSV to OG} for the
final statement).

If we denote by $G\subset\Mon^2(\widetilde{M})$ the subgroup generated by all these monodromy operators, in Section~\ref{section:Mon^2} we use lattice-theoretic results to prove that $G=\Or^+(\oH^2(\widetilde{M},\ZZ))$, thus completing the proof.

It follows from this argument that the monodromy group is generated by monodromy operators coming from projective families: this is a highly non-trivial feature. Even though the same statement is true for all the other known examples of irreducible holomorphic symplectic manifolds, it is not clear why this should hold in general.
\medskip

Finally, using recent developments in the theory of singular symplectic varieties, especially the work of Bakker and Lehn \cite{BakkerLehn:Singular2016}, we study the locally trivial monodromy group of the singular O'Grady moduli space $M_S$ (see Example~\ref{example:O'Grady moduli space}). 
\begin{thmno}[Theorem~\ref{thm:l t monodromy}]
Let $Y$ be a singular symplectic variety locally trivial deformation equivalent to $M_S$. Then $\Mon^2_{\operatorname{lt}}(Y)=\Ort^+(\oH^2(Y,\ZZ))$.
\end{thmno}
We finish by remarking that the same computation in the case of a singular symplectic variety of OG6 type is done in \cite{MongardiRapagnetta:MonodromyOG6}, where it is shown that in their case the locally trivial monodromy group is the whole group of orientation preserving isometries. The difference between the singular OG10 and the singular OG6 cases can be explained in terms of their singularities: in fact, singular moduli spaces of OG10 type can be either locally factorial or $2$-factorial (\cite[Theorem~1.1]{PeregoRapagnetta:Factoriality}), while singular moduli spaces of OG6 type are always $2$-factorial (\cite[Theorem~1.2]{PeregoRapagnetta:Factoriality}).

\subsection*{Acknowledgments}
It is my pleasure to thank Gregory Sankaran and Giovanni Mongardi for invaluable help, advice
and support. The author strongly benefitted from discussions with Valeria Bertini, Klaus Hulek, Antonio Rapagnetta, Giulia Sacc\`a and Ziyu Zhang. In particular, I thank Klaus Hulek and Radu Laza for explaining Proposition~\ref{prop:HLS} to me.
I also wish to thank Simon Brandhorst, who pointed out a mistake in an intermediate draft of this work, and Samuel Boissi\`ere and Alastair Craw for having read a first version of this work providing useful and helpful advice.
Important remarks and feedback arose
from attending the "Japanese--European Symposium on symplectic varieties and moduli spaces":
the author wishes to thank the organisers for this interesting meeting. Finally, I warmly thank the anonymous referee for the keen review of this manuscript: their comments and questions have surely improved it a lot.
Part of this work was carried on during the author's 
PhD program by the University of Bath. He wishes to thank the University of Bath and the EPSRC, the Riemann Centre in Hannover, the INdAM project for young researchers "Pursuit of IHS manifolds" and the Research Council of Norway project no. 250104 for financial and administrative support.

\subsection*{Notations}
By lattice we mean a free $\ZZ$-module $L$ together  with a non-degenrate symmetric
bilinear form $(\cdot,\cdot)\colon L\times L\to\ZZ$. We usually simply write $x^2$ for
$(x,x)$. We denote by $L(-1)$ the lattice obtained from $L$ by changing the sign of the 
bilinear form.

Since the bilinear form is non-degenerate, there is a canonical embedding
$L\subset L^*$, where $L^*=\Hom(L,\ZZ)$ is the dual lattice. The \emph{discriminant group} $A_L$ is the finite group $L^*/L$.
If $L=\oH^2(X,\ZZ)$ is the Beauville--Bogomolov--Fujiki lattice associated to an irreducible
holomorphic symplectic manifold $X$, then we simply write $A_X$ for the discriminant group.

The group of isometries of $L$ is denoted by $\Or(L)$, while $\Ort(L)$ denotes the 
subgroup of isometries that act as the identity on the discriminant group. 
If $M\subset L$ is a sublattice, we denote by $\Or(L,M)$ the subgroup of isometries $g$
such that $g(M)=M$. 

An \emph{isotropic element} is a vector $x\in L$ such that $x^2=0$.

Finally, $U$ will always denote the hyperbolic plane, i.e.\ the unique unimodular even lattice
of signature $(1,1)$; $A_2$, $E_8$ and $G_2$ denote the root lattices associated to the respective Dynkin diagrams.

%%%%%%%%%%%%%%%%%%%%%%%%%%%%%%%%%%%%%%%%%
%%%%%%%%%%%%%%%%%%%%%%%%%%%%%%%%%%%%%%%%%
\section{Preliminaries}\label{section:preliminaries}
\subsection{Irreducible holomorphic symplectic manifolds}
\begin{defn}
A compact K\"ahler manifold $X$ is called \emph{irreducible holomorphic symplectic} if it is 
simply connected and $\oH^0(X,\Omega^2_X)=\CC\sigma_X$, where $\sigma_X$ is non-degenerate at any point.
\end{defn} 
It follows directly from the definition that $\oH^2(X,\ZZ)$ is a torsion free $\ZZ$-module;
it turns out to be a lattice thanks to the Beauville--Bogomolov--Fujiki form $q_X$
(\cite{Beauville:c1=0}). This lattice structure is indispensable for studying the 
geometry of an irreducible holomorphic symplectic manifold $X$; we refer to 
\cite{GrossHuybrechtsJoyce:CalabiYau} and \cite{Markman:Survey} for a detailed account
of results on their geometry.

Let $X_1$ and $X_2$ be two irreducible holomorphic symplectic manifolds that are
deformation equivalent.
\begin{defn}\label{defn:parallel transport operators}
\begin{enumerate}
\item We say that $g\colon\oH^2(X_1,\ZZ)\to\oH^2(X_2,\ZZ)$ is a \emph{parallel transport operator} 
if there exists a smooth and proper family $p\colon\cX\to B$, points $b_1,b_2\in B$ and isomorphisms 
$\varphi_i\colon X_i\stackrel{\sim}{\longrightarrow}\cX_{b_i}$ such that the composition 
$(\varphi_2^*)^{-1}\circ g\circ\varphi_1^*$ is the parallel transport inside the local system 
$R^2 p_*\ZZ$ along a path $\gamma$ from $b_1$ to $b_2$. 
Here $R^2p_*\ZZ$ is endowed with the Gauss-Manin connection.
\item If $X_1=X_2=X$ and $\gamma$ is a loop, then the parallel transport is called
\emph{monodromy operator}. 
Such isometries form a group $\Mon^2(X)$ called \emph{monodromy group} (see \cite[Footnote~3 at page~3]{Markman:Survey}).
\end{enumerate}
\end{defn}
\begin{rmk}\label{rmk:orient pub}
For any irreducible holomorphic symplectic manifold $X$, let us denote by $\omega_X$ a fixed 
K\"ahler class. In the following we write $X$ as a pair $(M,I)$ where $M$ is a differential manifold and $I$ a complex structure. By Yau's solution to Calabi's conjecture (\cite{Yau:Calabi}) there is a hyper-K\"ahler metric $g$ on $M$ representing $\omega_X$, i.e.\ $\omega_X=\omega_I=g(I(-),-)$. Since the metric is hyper-K\"ahler, there exists another K\"ahler structure $J$ such that $I$, $J$ and $IJ=K$ satisfy the quaternionic relations. The symplectic form $\sigma_X$ is then defined as $\omega_J+i\omega_K$.

The positive (real) three-space $\langle \omega_I,\omega_J,\omega_K\rangle=\langle\omega_X,\Re(\sigma_X),\Im(\sigma_X)\rangle\subset
\oH^2(X,\RR)$ comes then with a preferred orientation (given by this basis).
We say that an isometry $\oH^2(X_1,\ZZ)\to\oH^2(X_2,\ZZ)$ is \emph{orientation preserving}
if it preserves this orientation. By definition, any parallel transport operator is orientation
preserving. In particular, if $\Or^+(\oH^2(X,\ZZ))$ denotes the group of orientation 
preserving isometries, then $\Mon^2(X)\subset\Or^+(\oH^2(X,\ZZ))$. We refer to \cite[Section~4]{Markman:Survey} for more details about this phenomenon.
\end{rmk}
\begin{example}\label{example:monodromy of K3}
An irreducible holomorphic symplectic manifold $S$ of dimension $2$ is a $K3$ surfaces. In this case $\Mon^2(S)=\Or^+(\oH^2(S,\ZZ))$ (\cite[Proposition~5.5 in Chapter~7]{Huybrechts:K3Surfaces}).
\end{example}

Now we recall the construction of two families of irreducible holomorphic symplectic manifolds.

%%%%%%%%%%%%%%%%%%%%%%%%%%%%%%%%%%%%%%%%%
\subsection{Moduli spaces of sheaves}\label{subsection:moduli spaces of sheaves}
Let $S$ be a projective $K3$ surface.
The cohomology ring 
$$\widetilde{\oH}(S,\ZZ)=\oH^0(S,\ZZ)\oplus\oH^2(S,\ZZ)\oplus\oH^4(S,\ZZ)$$  has a natural Hodge structure of weight $2$ given by putting $\widetilde{\oH}^{2,0}(S,\ZZ)=\oH^{2,0}(S,\ZZ)$, and a natural lattice structure that we now recall. If $v=(v_0,v_1,v_2)\in\widetilde{\oH}(S,\ZZ)$, then
$$ v^2=v_1^2-2v_0v_2,$$
where $v_1^2$ is the standard intersection product on $\oH^2(S,\ZZ)$. $\widetilde{\oH}(S,\ZZ)$ with this lattice structure is called the Mukai lattice. A vector $v\in\widetilde{\oH}(S,\ZZ)$ is called a Mukai vector if it is positive in the sense of \cite[Definition~0.1]{Yoshioka:ModuliSpacesAbelianSurfaces}.

In the following we work with a Mukai vector $v\in\widetilde{\oH}(S,\ZZ)$ such that
$v=2w$, where $w$ is primitive and $w^2=2$. Once the vector $v$ is fixed, there is a decomposition in $v$-chambers and $v$-walls of the ample cone of $S$, and any polarisation in the interior of a $v$-chamber is called $v$-generic (see \cite[Section~2.1]{PeregoRapagnetta:DeformationsOfO'Grady}).
For a $v$-generic polarisation $H\in\Pic(S)$, the moduli space $M_v(S)$ of $H$-semistable
sheaves on $S$ is singular exactly 
at those points corresponding to S-equivalence classes of strictly semistable sheaves. Let us denote
by $\Sigma_v$ the singular locus of $M_v(S)$.
By \cite{Mukai:SymplecticStructure}, the smooth locus of $M_v(S)$ has a symplectic form.
\begin{thm}[\protect{\cite{O'Grady:10dimPublished},\cite{Rapagnetta:BeauvilleFormIHSM},\cite{LehnSorger:Singularity},\cite{PeregoRapagnetta:DeformationsOfO'Grady}}]\label{thm:big OG10}
Let $v$ and $H$ as above.
\begin{enumerate}
\item There exists a symplectic desingularisation $\pi\colon\widetilde{M}_v(S)\to M_v(S)$ such
that $\widetilde{M}_v(S)$ is an irreducible holomorphic symplectic manifold of dimension $10$.
Moreover, $\widetilde{M}_v(S)$ is the blow up of $M_v(S)$ at $(\Sigma_v)_{\operatorname{red}}$.
\item Let $S'$ be another K3 surface, and choose a Mukai vector $v'$ and a $v'$-generic polarisation $H'$ as above. Then $\widetilde{M}_{v'}(S',H')$ is deformation equivalent to $\widetilde{M}_v(S,H)$. The deformation is obtained by deforming the surface $S'$ to $S$, the Mukai vector $v'$ to $v$ and the generic polarisation $H'$ to $H$ according to the notion of OLS-triple defined in \cite[Definition~2.12]{PeregoRapagnetta:DeformationsOfO'Grady}.
\item The pullback
$$ \pi^*\colon \oH^2(M_v(S),\ZZ)\to\oH^2(\widetilde{M}_v(S),\ZZ)$$
is injective. In particular $\oH^2(M_v(S),\ZZ)$ has a pure Hodge structure of weight $2$ and a lattice structure inherited by those of $\oH^2(\widetilde{M}_v(S),\ZZ)$. Moreover, there exists a Hodge isometry
$$v^\perp\stackrel{\sim}{\longrightarrow}\oH^2(M_v(S),\ZZ).$$ 
\item There is an isometry 
$$\oH^2(\widetilde{M}_v(S),\ZZ)\cong U^3\oplus E_8(-1)^2\oplus G_2(-1),$$
where $G_2(-1)=\left(\begin{array}{rr}-2 & 3 \\ 3 & -6 \end{array}\right)$.
\end{enumerate}
\end{thm}
Any irreducible holomorphic symplectic manifold that is deformation equivalent to a desingularised moduli space $\widetilde{M}_v(S,H)$ as in the Theorem above is said to be of OG10 type.

Define
\begin{equation}\label{eqn:Gamma_v in general}
\Gamma_v:=\left\{\left(\alpha,k\frac{\sigma}{2}\right)\in(v^\perp)^*\oplus\ZZ\frac{\sigma}{2}\mid k\in2\ZZ \Leftrightarrow \alpha\in v^\perp\right\}\subset v^\perp_{\QQ}\oplus\QQ\sigma,
\end{equation}
with the non-degenerate pairing $b\left((w_1,m_1\sigma),(w_2,m_2\sigma)\right)=
(w_1,w_2)-6m_1m_2$. Notice in particular that $\sigma^2=-6$. Moreover, $\Gamma_v$ has a natural Hodge structure as follows: $v^\perp$ has a Hodge structure inherited by $\widetilde{\oH}(S,\ZZ)$) and we declare $\sigma$ to be of type $(1,1)$.
\begin{thm}[\protect{\cite[Theorem~3.1]{PeregoRapagnetta:Factoriality}}]\label{thm:factoriality}
$\Gamma_v$ is a lattice Hodge isometric to $\oH^2(\widetilde{M}_v(S),\ZZ)$.
\end{thm}

\begin{example}[O'Grady moduli space]\label{example:O'Grady moduli space}
Let us fix $v=(2,0,-2)$. In this case, we use the short notation $\widetilde{M}_S$ and $M_S$
instead of $\widetilde{M}_{(2,0,-2)}(S)$ and $M_{(2,0,-2)}(S)$. The space $\widetilde{M}_S$ is called \emph{O'Grady moduli space}, while the space $M_S$ is called \emph{singular O'Grady moduli space} (cf.\ Example~\ref{example:singular moduli spaces}), since this is the case first studied by O'Grady in \cite{O'Grady:10dimPublished}.
The locus $B_S=M_S\setminus M^{\operatorname{lf}}_S$ of non-locally free sheaves is a Weil
divisor (\cite[Section~3.1]{O'Grady:10dimPublished}) and we denote by $\widetilde{B}_S$ its strict transform. We notice that, by \cite{Perego:2Factoriality}, $B_S$ is not Cartier, but $2B_S$ is. Since $\widetilde{M}_S$ is smooth, the divisor $\widetilde{B}_S$ is always Cartier and by \cite{Rapagnetta:BeauvilleFormIHSM}, 
$$\langle\widetilde{B}_S,\widetilde{\Sigma}_S\rangle=G_2(-1),$$ 
where $\widetilde{\Sigma}_S$ is the exceptional divisor of the desingularisation. Here $\langle\widetilde{B}_S,\widetilde{\Sigma}_S\rangle\subset\oH^2(\widetilde{M}_S,\ZZ)$ is the sublattice of the Beauville--Bogomolov--Fujiki lattice generated by $\widetilde{B}_S$ and $\widetilde{\Sigma}_S$. More precisely, Rapagnetta shows that $\widetilde{\Sigma}_S^2=-6$, $\widetilde{B}_S^2=-2$ and $(\widetilde{\Sigma}_S,\widetilde{B}_S)=3$.
Moreover, Rapagnetta also explicitly describes an isometry
$$\oH^2(\widetilde{M}_S,\ZZ)\cong\oH^2(S,\ZZ)\oplus\langle\widetilde{B}_S,\widetilde{\Sigma}_S\rangle,$$
where the inclusion of $\oH^2(S,\ZZ)$ in $\oH^2(\widetilde{M}_S,\ZZ)$ is the composition of the Donaldson's map and the pullback by the desingularisation $\pi_S\colon\widetilde{M}_S\to M_S$ (see also \cite[Section~5]{O'Grady:10dimPublished}). 

Finally, by \cite[Remark~3.2]{PeregoRapagnetta:Factoriality}, the isometry $\Gamma_{(2,0,-2)}\cong\oH^2(\widetilde{M}_S,\ZZ)$ is explicitly given by the
function 
\begin{equation}\label{eqn:Gamma_v}
\left(\left(\frac{n}{2},\xi,\frac{n}{2}\right),k\frac{\sigma}{2}\right)\mapsto\xi+n\widetilde{B}+\frac{n+k}{2}\widetilde{\Sigma}.
\end{equation}
For future reference, we notice that the particular case in which $k=0$ is exactly the isometry $v^\perp\cong\oH^2(M_S,\ZZ)$ in item $(3)$ of Theorem~\ref{thm:big OG10} in this case (loc.\ cit.).
\end{example}
\begin{rmk}\label{rmk:factoriality}
The singular O'Grady moduli space $M_S$ in Example~\ref{example:O'Grady moduli space} is $2$-factorial (\cite{Perego:2Factoriality}). In general, the factoriality of $M_v(S)$
depends on the vector $w$ such that $v=2w$: if there exists $u\in\widetilde{\oH}^{1,1}(S,\ZZ)$
such that $(u,w)=1$, then $M_v(S)$ is $2$-factorial; if $(u,w)\in2\ZZ$ for every
$u\in\widetilde{\oH}^{1,1}(S,\ZZ)$, then $M_v(S)$
is locally factorial (cf.\ \cite[Theorem~1.1]{PeregoRapagnetta:Factoriality}).
\end{rmk}

\begin{example}[Torsion sheaves]\label{exe:torsion sheaves}
Let $S$ be a projective K3 surface and $H$ a polarisation such that $H^2=2$. Then any vector of the form $v=(0,2H,b)$ gives a moduli space $M_v(S)$ of dimension $10$. From now on we always assume the polarisation to be $v$-generic. If $b$ is odd, the moduli space is smooth (\cite{Yoshioka:ModuliSpacesAbelianSurfaces}); if $b=2a$ is even, then we are in the situation of Theorem~\ref{thm:big OG10} and the singular moduli space $M_v(S)$ admits a symplectic desingularisation that is an irreducible holomorphic symplectic manifold of type OG10.

Let us then consider vectors of the form $(0,2H,2a)$. A general sheaf $F\in M_v(S)$ is of the form $i_*L$, where $i\colon C\to S$ is an embedding, $C\in|2H|$ is a smooth curve of genus $5$ and $L$ is a line bundle on $C$ of degree $2a+4$. In particular $M_v(S)$ contains the relative Picard variety $Pic^{2a+4}_U$, where $U\subset|2H|$ is the open subset parametrising smooth curves. 

There is a natural morphism 
$$p_v\colon M_v(S)\longrightarrow|2H|$$
defined by sending a sheaf to its Fitting support (see \cite[Section~1.4]{Mozgovyy:PhD}). In this way $M_v(S)$ can be thought as a projective compactification of $Pic^{2a+4}_U$.
Composing the morphism $p_v$ with the desingularisation morphism $\pi_v$ we get a morphism
$$
(p_v\circ\pi_v)\colon\widetilde{M}_v(S)\longrightarrow|2H|\cong\PP^5
$$
that is a Lagrangian fibration  (\cite{Matsushita:OnFibreSpace}).

Finally, if $\Pic(S)=\ZZ H$, then we can directly read the factoriality property of $M_v(S)$ in terms of the parity of $a$. Recall that $v=(0,2H,2a)$. By Remark~\ref{rmk:factoriality}, $M_v(S)$ is factorial if $a$ in even and $2$-factorial if $a$ is odd.
For example, the moduli space $M_{(0,2H,-4)}(S)$ is factorial, while the moduli space $M_{(0,2H,-2)}(S)$ is $2$-factorial (we will use this remark in Section~\ref{subsection:hyperelliptic}).
\end{example}

%%%%%%%%%%%%%%%%%%%%%%%%%%%%%%%%%%%%%%%%%
\subsection{Intermediate Jacobian fibrations}\label{subsection:J_V}
Let $V\subset\PP^5$ be a smooth cubic fourfold and $U\subset\PP(\oH^0(V,\cO(1))^*)$ the
open subset parametrising smooth linear sections. 

If $Y\in U$, the intermediate Jacobian of $Y$ is the principally polarised abelian variety defined by
$$ J_Y:=\oH^{2,1}(Y)^*/\oH_3(Y,\ZZ),$$
where $\oH_3(Y,\ZZ)\subset \oH^{2,1}(Y)^*$ by integration over cycles.

Running this construction relatively over $U$ yields an intermediate Jacobian fibration 
\begin{equation}\label{eqn:LSV over U}
\pi_U\colon\cJ_U\longrightarrow U
\end{equation}
that is Lagrangian with respect to a non-degenerate holomorphic closed $2$-form on $\cJ_U$
(\cite{DonagiMarkman:SpectralCovers}).
\begin{thm}[\cite{LSV:O'Gr10}]\label{thm:LSV}
Suppose that $V$ is very general. 
Then there exists a symplectic compactification 
$$ \pi_V\colon\cJ_V\longrightarrow\PP^5$$
of the intermediate Jacobian fibration (\ref{eqn:LSV over U}), such that $\cJ_V$ is an irreducible 
holomorphic symplectic manifold of OG10 type. Moreover, $\pi_V$ is a Lagrangian fibration.
\end{thm}
Here very general means outside a countable union of divisors in the moduli space of smooth cubic fourfolds.
The hypothesis of the statement above can be relaxed to general cubic fourfolds, that is to cubic fourfolds contained in a Zariski open subset of the moduli space of smooth cubic fourfolds (see for example \cite[Remark~4.2]{Brosnan:PerverseObstractions}). We notice though that recently G.\ Sacc\`a constructed an irreducible symplectic compactification of $\cJ_U$ for any smooth cubic fourfold (and even mildly singular cubic fourfolds), see \cite{Sacca:BirationalGeometryLSV}. Sacc\`a's compactification is obtained by using recent developments in the minimal model program; in particular it is not constructive and, a priori, not even unique.
%Remember that the moduli space of smooth cubic fourfolds contains countable many divisors corresponding to special cubic fourfolds (in the sense of Hassett, \cite{Hassett:SpecialCubicFourfolds}). Recently Sacc\`a extended this compactification to any smooth cubic fourfold, see \cite{Sacca:BirationalGeometryLSV}.

%%%%%%%%%%%%%%%%%%%%%%%%%%%%%%%%%%%%%%%
\subsection{Singular symplectic varieties}\label{section:Singular}
\begin{defn}%[\protect{Cf.\ \cite[Definition~3.2]{BakkerLehn:Singular2016}}]
A \emph{singular symplectic variety} $Y$ is a normal complex variety such that its regular locus $Y_{\operatorname{reg}}$ has a symplectic form that extends holomorphically to any resolution of singularities. 

A \emph{symplectic resolution of singularities} of $Y$ is a resolution of singularities $\pi\colon X\to Y$ such that the symplectic form on $Y_{\operatorname{reg}}$ extends holomorphically to a symplectic form on $X$. We say that $\pi\colon X\to Y$ is an \emph{irreducible symplectic resolution of singularities} if moreover $X$ is an irreducible holomorphic symplectic manifold.
\end{defn}
In the following we will only be interested in singular symplectic varieties having an irreducible symplectic resolution of singularities, and we refer to \cite{BakkerLehn:Singular2016} for results and background. 
%(for an up-to-date account of singular symplectic varieties without a symplectic resolution of singularieties and for further references of the subject, we refer to \cite{BakkerLehn:Singular2018}). 
The main and unique example we consider in this paper is the following.
\begin{example}\label{example:singular moduli spaces}
We use the same notations as in Section~\ref{subsection:moduli spaces of sheaves}.
Let $S$ be a projective $K3$ surface and $v=2w\in\widetilde{\oH}(S,\ZZ)$ a Mukai vector such that $w$ is primitive and $w^2=2$. For any choice of $v$-generic polarisation $H$, the moduli space $M_v(S)$ of $H$-semistable sheaves is a singular symplectic variety having an irreducible symplectic resolution of singularity by Theorem~\ref{thm:big OG10} (see also \cite{Mukai:SymplecticStructure} for the existence of the symplectic form on the regular part of $M_v(S)$). 

When the Mukai vector is $v=(2,0,-2)$, then $M_v(S)=M_S$ is the singular O'Grady moduli space of Example~\ref{example:O'Grady moduli space}.
\end{example}
If $Y$ is a singular symplectic variety with an irreducible symplectic resolution of singularities, then $\oH^2(Y,\ZZ)$ is endowed with a non-degenerate symmetric bilinear form turning it into a lattice of signature $(3,b_2(Y)-3)$. More precisely, if $\pi\colon X\to Y$ is the irreducible symplectic resolution, then by \cite[Lemma~3.5]{BakkerLehn:Singular2016} the pullback $\pi\colon\oH^2(Y,\ZZ)\to\oH^2(X,\ZZ)$ is injective and the restriction of the Beauville--Bogomolov--Fujiki form on $\oH^2(X,\ZZ)$ to $\oH^2(Y,\ZZ)$ is non-degenerate. Moreover, by the same result, the pullback is an isomorphism on the transcendetal part and the orthogonal complement of $\oH^2(Y,\ZZ)$ in $\oH^2(X,\ZZ)$ is negative definite, justifying the previous claim on the signature of $\oH^2(Y,\ZZ)$. This lattice structure is invariant under locally trivial deformations, according to the following definition.
\begin{defn}
A locally trivial family is a proper morphism $f\colon\cY\to T$ of complex analytic spaces such that, for every point $y\in\cY$, there exist open neighborhoods $V_y\subset \cY$ and $V_{f(y)}\subset T$, and an open subset $U_y\subset f^{-1}(f(y))$ such that 
$$V_y\cong U_y\times V_{f(y)}.$$
\end{defn}
If the morphism $f$ is smooth, i.e.\ the family is smooth, then the condition in the definition is trivially satisfied. The most important example for our purpose is the following.
\begin{example}\label{exe:family of OG moduli spaces}
Let $p\colon\cS\to T$ be a smooth family of projective $K3$ surfaces, and suppose that there exists a flat section $\cH\in R^2p_*\ZZ$ such that $\cH_t$ is a generic polarisation for every $t\in T$. Notice that, the second entry being $0$, the Mukai vector $(2,0,-2)$ extends to a flat section of the local system $\widetilde{R}p_*\ZZ$, whose stalk at $t\in T$ is the Mukai lattices $\widetilde{\oH}(\mathcal{S}_t,\ZZ)$. Then the relative O'Grady moduli space $f\colon\cM\to T$, whose fibres are the O'Grady moduli spaces $\cM_t=M_{\cS_t}$, is a locally trivial family (see \cite[Proposition~2.16]{PeregoRapagnetta:DeformationsOfO'Grady}).
\end{example}
\begin{defn}
Let $Y$ be a singular symplectic variety and $\pi\colon X\to Y$ a symplectic resolution of singularities.
\begin{enumerate}
\item The locally trivial monodromy group $\Mon^2(Y)_{\operatorname{lt}}$ of $Y$ is the subgroup of $\Or(\oH^2(Y,\ZZ))$ generated by isometries arising by parallel transport along loops in a locally trivial family of $Y$.
\item The monodromy group $\Mon^2(\pi)$ of the desingularisation $\pi$ is the subgroup of the product $\Or(\oH^2(X,\ZZ))\times\Or(\oH^2(Y,\ZZ))$ consisting of pairs $(g_1,g_2)$ such that $g_1\in\Mon^2(X)$ and $g_2\in\Mon^2(Y)_{\operatorname{lt}}$, and such that $g_1\circ\pi^*=\pi^*\circ g_2$.
\end{enumerate}
\end{defn}

%%%%%%%%%%%%%%%%%%%%%%%%%%%%%%%%%%%%%%%%%
%%%%%%%%%%%%%%%%%%%%%%%%%%%%%%%%%%%%%%%%%
\section{Monodromy operators coming from the family of O'Grady moduli spaces}\label{section:M_S}
Let $S$ be a projective $K3$ surface, $H$ a generic polarisation,
$M_S$ the O'Grady moduli space and $\widetilde{M}_S$ its symplectic desingularisation. 
We refer to Example~\ref{example:O'Grady moduli space} for notations.
In particular, we always denote by $G_2(-1)$ the lattice generated by the divisor 
$\widetilde{B}_S$ and the exceptional divisor $\widetilde{\Sigma}_S$.

Recall that
$$ \oH^2(\widetilde{M}_S,\ZZ)\cong\oH^2(S,\ZZ)\oplus G_2(-1). $$
In the following, $\Or^+(\oH^2(\widetilde{M}_S,\ZZ),G_2(-1))$ is the group of isometries fixing the sublattice $G_2(-1)$.
The restriction map
\begin{equation}\label{eqn:r M_S}
r\colon\Or^+(\oH^2(\widetilde{M}_S,\ZZ),G_2(-1))\longrightarrow\Or^+(\oH^2(S,\ZZ))
\end{equation}
is surjective and $\Or^+(\oH^2(S,\ZZ))=\Mon^2(S)$ (see Example~\ref{example:monodromy of K3}).
We want to show that, given a monodromy operator $g\in\Mon^2(S)$, there exists a canonical 
extension $\tilde{g}\in\Or^+(\oH^2(\widetilde{M}_S,\ZZ),G_2(-1))$ such that 
$\tilde{g}\in\Mon^2(\widetilde{M}_S)$. As we will see, this extension is given by the identity on 
$G_2(-1)$.
\medskip

Let $T$ be a curve and $(S,H)$ a polarised $K3$ surface. Let $\mathcal{S}_T\to T$ be a 
deformation family such that $\mathcal{S}_0=S$ for a base point $0\in T$ and 
let $\cH_T$ be a line bundle on $\mathcal{S}_T$, flat over $T$, such that $\cH_0=H$ (see Example~\ref{exe:family of OG moduli spaces}). 
It is known that the set of points $t\in T$ such that $\cH_t$ is not ample is finite.
Moreover, Perego and Rapagnetta notice in 
\cite[Lemma~2.6]{PeregoRapagnetta:DeformationsOfO'Grady} that the set of points
$t\in T$ such that $\cH_t$ is not generic is also finite. We summarise this remark in the following
statement for future reference.
\begin{lemma}\label{lemma:def O'Grady on curves}
Up to removing a finite number of points from $T$, we can suppose that $\cH_t$ is
ample and generic for every $t\in T$.
\end{lemma}
In the following we assume that $\cH_T$ satisfies the conditions of 
Lemma~\ref{lemma:def O'Grady on curves}.

Consider the relative moduli space $\mathcal{M}_T\to T$ (resp.\ $\mathcal{M}^s_T$) parametrising
rank $2$ semistable (resp.\ stable) sheaves on the fibres of $\mathcal{S}_T\to T$ with trivial 
determinant and second Chern class equal to $4$ (cf.\ 
\cite[Theorem~4.3.7]{HuybrechtsLehn:ModuliSpaces}). 
Notice that $\mathcal{M}^s_T\subset\mathcal{M}_T$ is open.

Since $\mathcal{M}_t$ is reduced and
irreducible for every $t\in T$, $\mathcal{M}_T$ is flat over $T$
(\cite[Proposition~II.2.19]{EisenbudHarris:GeometryOfSchemes} and
\cite[Lemma~2.15]{PeregoRapagnetta:DeformationsOfO'Grady}) and we can think of it as a 
deformation of (singular) moduli spaces. Now, define $\Sigma_T:=\mathcal{M}_T\setminus\mathcal{M}^s_T$.
As explained in the proof of \cite[Proposition~2.16]{PeregoRapagnetta:DeformationsOfO'Grady},
since $\cH_t$ is generic for every $t\in T$, $\Sigma_t$ is an irreducible closed subvariety which
coincides with the singular locus of $\mathcal{M}_t$. 
\begin{rmk}\label{rmk:modular description of Sigma}
Notice that $\Sigma_T$ has a modular description as the relative second symmetric product 
$\mathcal{S}ym^2\mathcal{S}_T^{[2]}$ (cf.\ first part of the proof of \cite[Proposition~2.16]{PeregoRapagnetta:DeformationsOfO'Grady}). In fact, by \cite[Lemma~1.1.5]{O'Grady:10dimPublished}, $\Sigma_t=\Sym^2\mathcal{S}_t^{[2]}$.  The singular locus of $\Sigma_T$ is then identified 
with $\mathcal{S}_T^{[2]}$.
This implies that $(\Sigma_{\operatorname{red}})_t=(\Sigma_t)_{\operatorname{red}}$ for every $t\in T$.
\end{rmk}
By \cite[Proposition~II.2.19]{EisenbudHarris:GeometryOfSchemes} we have that
$\Sigma_T$ and $(\Sigma_T)_{\operatorname{red}}$ are flat over $T$.
Blowing up $\mathcal{M}_T$ at $(\Sigma_T)_{\operatorname{red}}$ yields a projective and flat projection
\begin{equation}\label{eqn:def of OG10 from K3}
p\colon\widetilde{\mathcal{M}}_T\longrightarrow T
\end{equation} 
such that $\widetilde{\mathcal{M}}_t=\widetilde{M}_{\mathcal{S}_t}$.
Notice that a priori it is not obvious that the blow-up of the family is the family of the blow-ups: 
this follows from 
\cite[Proposition~2.16]{PeregoRapagnetta:DeformationsOfO'Grady}.

The family (\ref{eqn:def of OG10 from K3}) is the deformation family of O'Grady manifolds
associated to a deformation of polarised $K3$ surfaces.

The first remark is the following.
\begin{lemma}\label{lemma:monodromy preserves Sigma}
Let $\widetilde{M}_S$ be the O'Grady desingularisation of $M_S$ and $\widetilde{\Sigma}_S$ 
the exceptional divisor. Any monodromy operator $g$ arising from a deformation family
(\ref{eqn:def of OG10 from K3}), as constructed before, must satisfy the equality
$g(\widetilde{\Sigma}_S)=\widetilde{\Sigma}_S$.
\end{lemma}
\begin{proof}
This is clear from the discussion above. In fact, on $\widetilde{\mathcal{M}}_T$ there is the 
relative exceptional divisor $\widetilde{\Sigma}_T$ which is flat over $T$. The associated class in 
cohomology is then flat in the local system $R^2p_*\ZZ$ and hence preserved by any parallel transport
in the same local system.
\end{proof}
Next, we want to understand what is the orbit of the divisor $\widetilde{B}_S$ under monodromy 
operators arising from this kind of family. This is more subtle, because the locus 
$\mathcal{B}_T:=\mathcal{M}_T\setminus\mathcal{M}_T^{\operatorname{lf}}$ does not have a modular description as in 
Remark~\ref{rmk:modular description of Sigma}. Here and in the following 
$\mathcal{M}_T^{\operatorname{lf}}\subset\mathcal{M}_T$ is
the open subset parametrising locally free sheaves on the fibres of $\mathcal{S}_T\to T$.
We need to work with the Uhlenbeck compactification $\overline{\mathcal{N}}_\infty$ 
of the Donaldson--Yau moduli space $\mathcal{N}_\infty$ of anti-self-dual connections on 
the principal bundle of rank $2$, trivial determinant and second Chern class of degree $4$ on the differentiable manifold underlying $S$ (\cite{FriedmanMorgan:SmoothFourManifolds}). 
Recall that
$\overline{\mathcal{N}}_\infty$ exists as a (reduced) projective scheme and there is
a regular morphism of schemes
$$\phi\colon M_S\longrightarrow\overline{\mathcal{N}}_\infty.$$
Moreover, 
$\overline{\mathcal{N}}_\infty=\mathcal{N}_\infty\coprod S^{(4)}$,
where $S^{(4)}$ stands for the fourth symmetric product of $S$ (see discussion after \cite[Proposition~3.1.1]{O'Grady:10dimPublished}). 
The morphism $\phi$ restricts to an isomorphism 
$M^{\operatorname{lf}}_S\cong\mathcal{N}_\infty$ (\cite{Li:AlgebroGeometric}).

We want to relativise this construction to the family $p\colon\mathcal{M}_T\to T$. 
For this, we need to run the same arguments as in 
\cite[Section~1, Section~2]{Li:AlgebroGeometric} in families.

Let $\operatorname{Quot}_{\mathcal{S}/T}$ be a Quot scheme of sheaves on the fibres of 
$\mathcal{S}_T\to T$, whose open subset $Q_T\subset\operatorname{Quot}_{\mathcal{S}/T}$ parametrising semistable points with respect to the action of the algebraic 
group $G=\operatorname{PGL}(N)$ (for a suitable integer $N$) has $\mathcal{M}_T$ as GIT quotient. On $\mathcal{S}_T\times Q_T$ there is a universal quotient sheaf
$F_T$, flat over $T$ (cf.\ \cite[Theorem~2.2.4]{HuybrechtsLehn:ModuliSpaces}). 

Now let $k\geq1$ and $D_T\in|k\cH_T|$ be a divisor which is smooth over $T$. 
Notice that such a divisor $D_T$ always exists, up to shrinking the base $T$.
Since the fibres of $D_T$ over
$T$ are smooth algebraic curves, we can consider the relative Jacobian $J^{g(D_T)-1}(D_T)$.
Here $g(D_T)$ means the genus of the general fibre of $D_T$ over $T$. 
Let $\theta_{D_T}\in J^{g(D_T)-1}(D_T)$ be flat over $T$.
Then we define the line bundle
\begin{equation}
\widetilde{\mathcal{L}}_k(D_T,\theta_{D_T}):=\det\left(R^\bullet q_{1*}(F_T|_{D_T}\otimes q_2^*\theta_{D_T})\right)^{-1}
\end{equation}
where $q_i$ is the projection from $Q_T\times D_T$ to the $i$-th factor and $F_T|_{D_T}$ is the restriction of $F_T$ to
$Q_T\times D_T$.
Notice that by construction $\widetilde{\mathcal{L}}_k(D_T,\theta_{D_T})$ is flat over $T$.

\begin{lemma}
$\widetilde{\mathcal{L}}_k(D_T,\theta_{D_T})$ descends to a line bundle $\mathcal{L}_k(D_T,\theta_{D_T})$
on $\mathcal{M}_T$.
\end{lemma}
\begin{proof}
Since $\mathcal{M}_T$ is constructed as a $G$-quotient from $Q_T$,  
Kempf's criterion (\cite[Proposition~2.3]{DrezetNarasimhan:GroupDePicard}) says that a $G$-bundle $E$ on $Q_T$
descends to $\mathcal{M}_T$ if and only if for every closed point $x\in Q_T$ with closed orbit, the 
stabiliser $G_x$ acts trivially on $E_x$. First we show that $\widetilde{\mathcal{L}}_k(D_T,\theta_{D_T})$ is 
a $G$-bundle. As in the first part of the proof of \cite[Proposition~1.7]{Li:AlgebroGeometric}, we only need to show that the subgroup $\CC^*\subset\operatorname{GL}(N)$ acts trivially on $\widetilde{\mathcal{L}}_k(D_T,\theta_{D_T})$. If $\alpha\in\CC^*$, the homothety $\alpha.\id$ acts fibrewise on $F_T$. On the other hand, our choice of $D_T$ is such that the Euler characteristic of the fibres over $T$ is zero, so that $\alpha.\id$ induces the identity on $\widetilde{\mathcal{L}}_k(D_T,\theta_{D_T})$.

Now, closed points in $Q_T$ are all of the form
$i_{t*}F_t$, where $i_t\colon\mathcal{S}_t\to\mathcal{S}_T$ is the inclusion and $F_t\in\mathcal{M}_t$. 
Moreover, since $\cH_t$ is assumed to be generic for every $t\in T$, $Q_t$ satisfies the hypotheses 
of \cite[Proposition~1.7]{Li:AlgebroGeometric} and therefore the proof is reduced to the proof of
\cite[Proposition~1.7]{Li:AlgebroGeometric}.
\end{proof}

With an abuse of notation, we denote by $\mathcal{L}_k$ the line bundle $\mathcal{L}_k(D_T,\theta_{D_T})$.
\begin{prop}
Let $(\mathcal{S}_T,\cH_T)$ be a polarised family of $K3$ surfaces over a curve $T$. 
Let $\mathcal{L}_k=\mathcal{L}_k(D_T,\theta_{D_T})$ be the line bundle on $\mathcal{M}_T$ constructed 
above and suppose $k\geq9$. 
Then there exists a positive integer $\bar{m}$ such that $(\mathcal{L}_k^m)_t$ is generated by
global sections for every $t\in T$ and for every $m\geq\bar{m}$.
\end{prop}
\begin{proof}
For any $t\in T$, there exists a positive integer $m_t$ such that
$(\mathcal{L}_{k}^{m_t})_t$ is generated by global sections for every $m\geq m_t$ and $k\geq9$
(\cite[Theorem~3]{Li:AlgebroGeometric}).
Now, define 
$$ \bar{m}:=\sup_{t\in T}\{m_t\} $$
and since $T$ is quasi-compact (in the Zariski topology), $\bar{m}<\infty$.
\end{proof}
The pushforward $p_*\mathcal{L}_k^{\bar{m}}$ is not locally free in general, but its double dual 
$p_*(\mathcal{L}_k^{\bar{m}})^{\vee\vee}$ is always locally free 
(\cite[Corollary~1.4]{Hartshorne:ReflexiveSheaves}). The proposition above says then that the induced
map
\begin{equation}\label{eqn:Li in families}
\varphi_T\colon\mathcal{M}_T\longrightarrow\PP\left(p_*(\mathcal{L}_k^{\bar{m}})^{\vee\vee}\right)
\end{equation}
is a regular morphism of schemes. Notice that $\PP\left(p_*(\mathcal{L}_k^{\bar{m}})^{\vee\vee}\right)$ 
is flat over $T$ (it is a projective bundle) and that $\varphi_T$ is defined fibrewise.

Let us define $\overline{\mathcal{N}}_T$ as the image of $\mathcal{M}_T$ via $\varphi_T$. By construction 
(or by \cite[Proposition~II.2.19]{EisenbudHarris:GeometryOfSchemes}) $\overline{\mathcal{N}}_T$ is flat over $T$
and, for every $t\in T$, $\overline{\mathcal{N}}_t$ is the Donaldson--Uhlenbeck--Yau moduli space associated to the
$K3$ surface $\mathcal{S}_t$ (\cite[Theorem~5]{Li:AlgebroGeometric}). The natural projection
\begin{equation}\label{eqn:family of DUY}
\overline{\mathcal{N}}_T\longrightarrow T
\end{equation}
is then a family of Donaldson--Uhlenbeck--Yau moduli spaces.
If we put $\mathcal{N}_T=\varphi_T(\mathcal{M}_T^{\operatorname{lf}})$, then
we get a relative Uhlenbeck decomposition
$$\overline{\mathcal{N}}_T=\mathcal{N}_T\coprod\mathcal{S}_T^{(4)}$$
where $\mathcal{S}_T^{(4)}$ is the relative symmetric product, i.e.\ 
$(\mathcal{S}_T^{(4)})_t=\mathcal{S}_t^{(4)}$.
\begin{rmk}\label{rmk:symmetric power is flat}
Notice that $\mathcal{S}_T^{(4)}$ is flat over $T$.
\end{rmk}
\begin{rmk}
The construction above is not canonical: it depends on the choice of both $D_T$ and $\theta_{D_T}$,
so one should really write $\varphi_{T,D_T,\theta_{D_T}}$. 
Nevertheless, we suppress this dependence from the notation for the sake of clarity.

Anyway, when $T=\operatorname{Spec}(\CC)$ is a point, 
Li noticed that $\mathcal{L}_k(D_T,\theta_{D_T})$ 
does not depend on $D_T$ and $\theta_{D_T}$. In particular, for a general base $T$, the claim is true fibrewise and so, if $D'_T$ is another smooth divisor
on $\mathcal{M}_T$ and $\theta_{D'_T}\in J^{g(D'_T)-1}(D'_T)$, then 
$$\mathcal{L}_k(D_T,\theta_{D_T})\cong\mathcal{L}_k(D'_T,\theta_{D'_T})\otimes p^*A $$
where $A$ is a line bundle on $T$. 
\end{rmk}

\begin{prop}\label{prop:B is flat}
$\mathcal{B}_T=\mathcal{M}_T\setminus\mathcal{M}_T^{\operatorname{lf}}$ is flat over $T$.
\end{prop}
\begin{proof}
Consider the surjective morphism
$$\varphi_T\colon\mathcal{B}_T\longrightarrow\mathcal{S}_T^{(4)}$$
obtained by restricting the morphism (\ref{eqn:Li in families}). 
By \cite[Proposition~II.2.19]{EisenbudHarris:GeometryOfSchemes}, it is enough to show that there are
no embedded components of $\mathcal{B}_T$ supported on a fibre $\mathcal{B}_t$, for every $t\in T$.
Suppose such a component $\overline{\mathcal{B}}_T\subset\mathcal{B}_T$ exists and is supported
on the fibre $\mathcal{B}_{t_0}$. 
Since $\varphi_T$ is defined fibrewise, 
$\varphi_T(\overline{\mathcal{B}}_T)=\varphi_{t_0}(\overline{\mathcal{B}}_T)\subset\mathcal{S}_{t_0}^{(4)}$. Since $\mathcal{S}_T^{(4)}$ is flat over $T$ (Remark~\ref{rmk:symmetric power is flat}), 
$\varphi_{t_0}(\overline{\mathcal{B}}_T)$ cannot be an embedded component of 
$\mathcal{S}_T^{(4)}$. 

We conclude that such an embedded component $\overline{\mathcal{B}}_T$ cannot exist and that
$\mathcal{B}_T$ is flat over $T$.
\end{proof}
\begin{lemma}\label{lemma:monodromy preserves B}
Let $\widetilde{M}_S$ be the O'Grady desingularisation of $M_S$, $B_S\subset M_S$ the
divisor of non-locally free sheaves and $\widetilde{B}_S$ its strict transform. 
Any monodromy operator $g$ arising from a deformation family
(\ref{eqn:def of OG10 from K3}) must satisfy the equality
$g(\widetilde{B}_S)=\widetilde{B}_S$.
\end{lemma}
\begin{proof}
It follows directly from Proposition~\ref{prop:B is flat} as in the proof of 
Lemma~\ref{lemma:monodromy preserves Sigma}.
\end{proof}

The main result of this section is the following proposition.
\begin{thm}\label{thm:my operators O'Grady}
Let $g\in\Or^+(\oH^2(\widetilde{M}_S,\ZZ))$ be such that $g(\widetilde{\Sigma}_S)=\widetilde{\Sigma}_S$, $g(\widetilde{B}_S)=\widetilde{B}_S$ and $g(H)=H$. Then $g$ is a monodromy operator.
\end{thm}
\begin{proof}
Let $g$ be as in the statement. In particular $g\in\Or^+(\oH^2(\widetilde{M}_S,\ZZ),G_2(-1))$ and so
its image $r(g)$ under the restriction map (\ref{eqn:r M_S}) is a monodromy operator on $S$ that preserves the polarisation of $S$ (i.e.\ it is a polarised monodromy operator -- cf.\ \cite[Corollary~7.4]{Markman:Survey}). 
This means that there exists a family of deformations $\mathcal{S}_T\to T$ such that $r(g)$ is obtained
by parallel transport along a loop $\gamma$ in $T$ centred in a point $0\in T$ corresponding
to $S$, and a family $\{\cH_t\}_{t\in T}$ of ample line bundles such that $\cH_0=H$.
By the proof of 
\cite[Proposition~5.5 in Chapter~7]{Huybrechts:K3Surfaces}, it follows that $T$ can be taken to be a curve (in fact, $T$ is a general pencil). Notice that in the proof of \cite[Proposition~5.5 in Chapter~7]{Huybrechts:K3Surfaces} the $K3$ is taken non projective, but exactly the same argument applies for projective surfaces.

The number of points $t\in T$ where $\cH_t$ is not generic is finite 
(cf.\ Lemma~\ref{lemma:def O'Grady on curves}). Let us denote by $T'$ the complement in $T$ of
these points. By \cite[Th{\'e}or{\`e}me~2.3 in Chapter~X]{Godbillon:AlgebraicTopology}, the induced map
$$ \pi_1(T',0)\longrightarrow\pi_1(T,0) $$
is surjective and so we can assume that $[\gamma]\in\pi_1(T',0)$.

By construction the parallel transport along $\gamma$ in the family
$\widetilde{\mathcal{M}}_{T'}\to T'$ is an isometry $g'$ such that $r(g')=r(g)$ and moreover,
by Lemma~\ref{lemma:monodromy preserves Sigma} and 
Lemma~\ref{lemma:monodromy preserves B}, $g'(\widetilde{\Sigma}_S)=\widetilde{\Sigma}_S$ 
and $g'(\widetilde{B}_S)=\widetilde{B}_S$. Therefore $g=g'$.
\end{proof}
\begin{rmk}
Since the family $\mathcal{M}_T\to T$ is locally trivial (cf.\ Section~\ref{section:Singular}), Proposition~\ref{prop:B is flat} implies that any (locally trivial) monodromy operator arising from this family must preserve the divisor $2B$ (notice that, as recalled in Example~\ref{example:O'Grady moduli space}, $B$ is not Cartier, while $2B$ is, and this property is preserved in this family). Moreover, running the same proof of Theorem~\ref{thm:my operators O'Grady} in this situation yields the analogous statement
$$\Mon^2(S)_{H}=\Or^+(\oH^2(S,\ZZ))_{H}\subset\Mon^2(M_S)_{\operatorname{lt}}.$$
\end{rmk}

%%%%%%%%%%%%%%%%%%%%%%%%%%%%%%%%%%%%%%%%%
%%%%%%%%%%%%%%%%%%%%%%%%%%%%%%%%%%%%%%%%%
\section{Exceptional reflections from singular moduli spaces}\label{section:exceptional reflections}
In this section we show the existence of monodromy reflections coming from exceptional
divisors on the singular moduli space $M_v(S)$.

First of all, let us explain what we mean by exceptional reflection. If $L$ is a lattice and $x\in L$ is an element, then the reflection $R_x$ is the rational isometry (i.e.\ an isometry of $L\otimes\QQ$) defined by 
$$ y\mapsto y-2\frac{(x,y)}{x^2}x.$$
If now $L=\oH^2(X,\ZZ)$ and $x$ is the class of a prime exceptional divisor, i.e.\ a reduced and irreducible divisor such that $x^2<0$ (cf.\ \cite[Definition~5.1]{Markman:Survey}), then Markman proved that the reflection $R_x$ is a monodromy operator (\cite[Theorem~1.1]{Markman:PrimeExceptional}. In particular this implies that $R_x$ is integral, imposing strong conditions on the numerical invariants of the divisor class $x$. Geometrically, in the projective case, we remark that prime exceptional divisors are exceptional divisors of divisorial contractions, possibly after a birational isomorphism of the ambient space (\cite[Proposition~1.4]{Druel:QuelquesRemarques}). More precisely, if $D$ is a prime exceptional divisor, then there exists another irreducible holomorphic symplectic manifold $X'$, a birational isomorphism $X\cong X'$ and a divisorial contraction $X'\to Y$, with $Y$ normal and projective, such that the image of $D$ in $X'$ coincides with the exceptional divisor of $X'\to Y$.
\medskip

We will use the notations and results introduced in Section~\ref{subsection:moduli spaces of sheaves}. In particular, $S$ is a projective $K3$ surface, $v=2w$ a Mukai vector such that $w^2=2$, and $H$ a $v$-generic polarisation. Then we consider the moduli space $M_v(S)$ and its irreducible symplectic desingularisation $\pi\colon\widetilde{M}_v(S)\to M_v(S)$ (see Theorem~\ref{thm:big OG10}). 

Let $D$ be a reduced and irreducible Cartier divisor on $M_v(S)$ and, with an abuse of notation, let us keep denoting by $D$ its cohomology class. Assume that $D^2=-2$. 
By \cite[Theorem~5.3]{MeachanZhang:BirationalGeometry}, $D$ arises as the exceptional 
divisor of a divisorial contraction in some birational model 
(which is a moduli space of Bridgeland stable objects on $S$). 
Note that, by \cite[Proposition~2.3]{MeachanZhang:BirationalGeometry} and the Cone Theorem 
\cite[Theorem~3.7]{KollarMori:BirationalGeometry}, $D$ is uniruled.
%Let $\pi\colon\widetilde{M}_v(S)\to M_v(S)$ be the symplectic resolution of singularities and $\widetilde{D}$ the strict transform of $D$. 
Let $\widetilde{D}$ be the strict transform of $D$ via the morphism $\pi\colon\widetilde{M}_v(S)\to M_v(S)$.
By \cite{LehnPacienza:MMP}, the minimal model program for the pair $(\widetilde{M}_v(S),\pi^*D)$
terminates, and the termination does not depend on the order of the contracted curves. Since $\widetilde{D}$ is negative and uniruled, being the strict transform of a uniruled divisor, the first step in the minimal model program above contracts all the rational curves in $\widetilde{D}$ obtaining a symplectic variety $Y$ and a divisorial contraction $\widetilde{M}_v(S)\to Y$, where $\widetilde{D}$ is the contracted divisor.
It follows that $\widetilde{D}$ is a prime exceptional divisor, and so the reflection $R_{\widetilde{D}}$ is of monodromy type. 
If $\widetilde{\Sigma}$ is the exceptional divisor of the desingularisation $\pi$, then 
$$\pi^*D=\widetilde{D}+m\widetilde{\Sigma}$$
is still of degree $-2$ (see item (3) of Theorem~\ref{thm:big OG10}). Since the reflection $R_{\widetilde{D}}$ is integral, this forces
$\widetilde{D}^2=-2$ and $m=0$. Let us summarise this remark in the
following result.
\begin{prop}
If $D\in \oH^2(M_v(S),\ZZ)$ is the class of a reduced and irreducible divisor such that $D^2=-2$, then the 
reflection $R_{\widetilde{D}}$ around the strict transform of $D$ is a monodromy operator.
Moreover $\widetilde{D}=\pi^*D$ and $\widetilde{D}^2=-2$.
\end{prop}

\begin{rmk}
Notice that if the Cartier divisor $D$ is not reduced, then the result is false. In fact, on the
O'Grady moduli space $M_S$ (cf.\ Example~\ref{example:O'Grady moduli space} for notation and results) the divisor $D=2B$ is not reduced and its strict transform $2\widetilde{B}=\pi^*(2B)-\widetilde{\Sigma}$ has degree $-8$.
\end{rmk}

\begin{example}\label{example:R_L}
Let $S$ be a projective $K3$ surface of genus $2$ with polarisation $H$, and suppose that
$\Pic(S)=\ZZ H+\ZZ K$ such that $K^2=-2$ and $(H,K)=0$. 
Let $v=(2,0,-2)$ be the Mukai vector of the O'Grady moduli space
and consider the class $(-1,H+K,-1)\in v^\perp$, that defines an irreducible and effective
divisor $Z$ on $M_S$ via the Hodge isometry in item $(3)$ of Theorem~\ref{thm:big OG10}. By construction $Z^2=-2$, so that by the proposition above $\widetilde{Z}=\pi^*Z$, and by the isometry (\ref{eqn:Gamma_v}) in Example~\ref{example:O'Grady moduli space},
$$\widetilde{Z}=H+K-2\widetilde{B}-\widetilde{\Sigma}.$$
\end{example}

%%%%%%%%%%%%%%%%%%%%%%%%%%%%%%%%%%%%%%%%%
%%%%%%%%%%%%%%%%%%%%%%%%%%%%%%%%%%%%%%%%%
\section{Monodromy operators coming from the family of intermediate Jacobian fibrations}\label{section:J_V}

In this section $V$ is a generic cubic fourfold, $\pi_V\colon\cJ_V\to\PP^5$ is the compactified 
intermediate Jacobian fibration.
Let $\Theta$ be any relative theta divisor on $\cJ_U$, rigidified along the zero section, 
and let $\overline{\Theta}$ be its closure in $\cJ_V$. Denote by $b_V=\pi_V^*\cO_{\PP^5}(1)$
the class of the fibration $\pi_V$. The following result was communicated to us by Klaus Hulek and Radu Laza.
\begin{prop}[Hulek--Laza]\label{prop:HLS}
\begin{enumerate}
\item The lattice $U_V=\langle\overline{\Theta},b_V\rangle$, generated by the relative theta 
divisor and the class of the fibration, is a hyperbolic plane.
\item There exists an isogeny of Hodge structures\footnote{Given two Hodge structures $H$ and $H'$, an isogeny between them is a morphism $\alpha\colon H\to H'$ such that $\alpha_{\QQ}\colon H\otimes\QQ\to H'\otimes\QQ$ is an isomorphism of Hodge structures.}
\begin{equation}\label{eqn:HLS}
\alpha\colon\oH^4(V,\ZZ)_{\operatorname{prim}}\longrightarrow U_V^\perp\subset\oH^2(\cJ_V,\ZZ)
\end{equation}
and an integer $N>0$ such that $x.y=-Nq_V(\alpha(x),\alpha(y))$ 
for any $x,y\in\oH^4(V,\ZZ)_{\operatorname{prim}}$.
\end{enumerate}
\end{prop}
With an abuse of notation we denote by $q_V$ the Beauville--Bogomolov--Fujiki form on $\cJ_V$.
\begin{proof}
The first claim follows from the Fujiki formula
\cite[Theorem~4.7]{Fujiki:Constant} with Fujiki constant $945$ (\cite{Rapagnetta:BeauvilleFormIHSM}). In fact, for dimension reasons one gets $b_V^{10}=0$, and so $q_V(b_V)=0$; on the other hand, applying the Fujiki formula again to the class 
$\overline{\Theta}+tb_V$ and taking the coefficient of the term $t^5$, one gets 
$q_{V}(\overline{\Theta},b_V)=(\overline{\Theta}^5.b_V^5)/5!=1$, where the last equality 
follows from Poincar\'e formula.
This is enough to conclude that $U_V$ is a hyperbolic plane, despite the value of 
$q_{V}(\overline{\Theta})$.

For the second claim, let us first construct the map $\alpha$. 
By \cite[Lemma~1.1]{LSV:O'Gr10}, there is a distinguished rational cycle 
$Z\in\CH^2(\cJ_V\times_{\PP^5}\cY_V)_{\QQ}$, where $\cY_V\to\PP^5$ is the universal family
of linear sections of $V$. The associated correspondence is $[Z]_*\colon\oH^4(\cY_V,\QQ)\longrightarrow\oH^2(\cJ_V,\QQ)$ defined by $[Z]_*(x)=\pi_{1*}(\pi_2^*x.Z)$, where $\pi_1$ and $\pi_2$ are the projection from $\cJ_V\times_{\PP^5}\cY_V$ to the two factors.
Denote by $q\colon\cY_V\to V$ the map that is the inclusion on 
each linear section (notice that $q$ is a $\PP^4$-bundle). Then 
$$\alpha':=[Z]_*\circ q^*\colon\oH^4(V,\QQ)\longrightarrow\oH^2(\cJ_V,\QQ)$$ 
is a morphism of rational Hodge structures by construction.
If $V$ is very general and $h\in\oH^2(V,\QQ)$ is the hyperplane section, then 
$\alpha'(h^2)\in (U_V)_{\QQ}$, so the same must hold for generic $V$. In particular, the restriction
$$\alpha\colon\oH^4(V,\QQ)_{\operatorname{prim}}\longrightarrow (U_V^\perp)_{\QQ}$$
is a well-defined morphism of rational Hodge structures.
Now, since $V$ is general, the Hodge structure on $\oH^4(V,\QQ)_{\operatorname{prim}}$ 
is irreducible and, since $\alpha$ is non-zero, $\alpha$ is an isomorphism of rational Hodge structures.

Also, since the lattices 
$\oH^4(V,\ZZ)_{\operatorname{prim}}$ and $U_V^\perp$
are anti-isometric, and $\alpha$ sends isotropic classes to isotropic classes, the $\QQ$-linear 
extensions of the symmetric bilinear pairing must coincide.

Finally, by clearing the denominators of $Z$, one gets an integral cycle. Moreover, notice that $\alpha'(h^2)\in U_V$, since $U_V\subset\oH^2(\cJ_V,\ZZ)$ is saturated (it is a hyperbolic plane by the previous point). Then the map $\alpha$
restricts to an isogeny of integral Hodge structures and the lattice structures are preserved up to 
a constant.
\end{proof} 
The constant $N$ comes from the fact that the cycle $Z$ is rational; there is no reason to
expect that $Z$ is integral. 

Now, we want to recall the construction of a distinguished theta divisor. Recall that $U\subset\PP\oH^0(V,\cO_V(1))^*$ is the open subset parametrising smooth linear sections. For $u\in U$, we denote by $Y_u\subset V$ the corresponding smooth cubic threefold.
Let $\cF_U$ be the relative Fano surface of lines, that is $\cF_u=F({Y_u})$ for any $u\in U$.
Consider the difference morphism
\begin{equation}\label{eqn:difference map}
f_V\colon\cF_U\times_U\cF_U\longrightarrow\cJ_U,
\end{equation}
defined fibrewise by sending two lines to the Abel--Jacobi invariant of their difference.
By \cite[Theorem~13.4]{ClemensGriffiths:CubicThreefolds}, the image (with reduced scheme structure) of this map is a relative theta divisor.
We denote by $\Theta_V$ the closure of the image of $f_V$. Notice that this is an effective
divisor and that, by \cite[Proposition~5.3, Theorem~5.7]{LSV:O'Gr10}, it is relatively ample on 
$\cJ_V$.
We will need the following result.
\begin{lemma}[\protect{\cite[Theorem~2]{Sacca:BirationalGeometryLSV}}]\label{lemma:theta exceptional}
$\Theta_V$ is a prime exceptional divisor. In particular its degree is $-2$ and the reflection
$R_{\Theta_V}$ is a monodromy operator.
\end{lemma}

The hyperbolic plane $U_V$ has thus a distinguished basis given by $b_V$ and $\Theta_V$.
By Theorem~\ref{thm:LSV}, the Beauville--Bogomolov--Fujiki lattice structure on $\oH^2(\cJ_V,\ZZ)$ is abstractly isometric to the one in item $(4)$ of Theorem~\ref{thm:big OG10}, and by item $(1)$ in Proposition~\ref{prop:HLS}, we write
$$ \oH^2(\cJ_V,\ZZ)\cong U_V\oplus U^{2}\oplus E_8(-1)^{2}\oplus G_2(-1).$$

The restriction map 
$$
r\colon\Or^+\left(\oH^2(\cJ_V,\ZZ),U_V\right)\longrightarrow\Or^+(U_V^\perp)
$$
is surjective and, by Proposition~\ref{prop:HLS} and \cite[Th{\'e}or{\`e}me~2]{Beauville:MonodromyHypersurfaces}, 
$$ \Ort^+(U_V^\perp)\cong\Ort^+(\oH^4(V,\ZZ)_{\operatorname{prim}})=\Mon^4(V). $$
Notice that Beauville shows that $\Mon^4(V)=\Or^+(\oH^4(V,\ZZ))_{h^2}$, where the latter is the group of isometries $g$ such that $g(h^2)=h^2$. By the last part of the proof of \cite[Th{\'e}or{\`e}me~2]{Beauville:MonodromyHypersurfaces} though, it follows that $\Or^+(\oH^4(V,\ZZ))_{h^2}=\Ort^+(\oH^4(V,\ZZ)_{\operatorname{prim}})$ as claimed.

Let 
$\cU\subset\PP(\oH^0(\PP^5,\cO(3))^*)$ be the parameter space of smooth cubic fourfolds.
We denote by $\cU'\subset\cU$ the open subset of non-special cubic fourfolds, so in particular
$\cU'$ is the complement in $\cU$ of the union of countably many divisors. 
By Theorem~\ref{thm:LSV}, there is a family
$$ \upsilon\colon\cJ_{\cU'}\longrightarrow\cU'$$
of intermediate Jacobian fibrations.

\begin{rmk}\label{rmk:Theta and b}
Notice that $\upsilon\colon\cJ_{\cU'}\to\cU'$ is a family of Lagrangian fibrations, and therefore any 
monodromy operator aring from this family must preserve $b_V$. Moreover, since $\Theta_V$ is a prime exceptional divisor, the class $\Theta_V$ of the theta divisor must be preserved as well.
\end{rmk}

\begin{prop}\label{thm:my monodromy}
Let $g\in\Ort^+(\oH^2(\cJ_V,\ZZ))$ be such that $g(\Theta_V)=\Theta_V$ and
$g(b_V)=b_V$. Then $g$ is a monodromy operator.
\end{prop}
\begin{proof}
Let $g$ be as in the statement; in particular $g\in\Ort^+\left(\oH^2(\cJ_V,\ZZ),U_V\right)$. 
Then, its restriction $r(g)$ induces the isometry
$\tilde{g}\in\Ort^+\left(\oH^4(V,\ZZ)_{\operatorname{prim}}\right)=\Mon^4(V)$. 
Then there exists a loop $\gamma$ in $\cU$ such that $\tilde{g}$ is the parallel transport along 
$\gamma$. The base $\cU\subset\PP^{55}$ is open in the standard topology. 
The restriction to $\cU$ of the Fubini--Study metric on $\PP^{55}$ can be non-complete on $\cU$: we can make such a metric complete by multiplying it with a smooth (scalar) function which diverges to infinity at least quadratically when approaching the boundary of $\cU$. Lemma~\ref{lemma:U'} below ensures that the natural map
$\pi_1(\cU')\to\pi_1(\cU)$ is surjective: we can move $\gamma$ away from 
special cubic fourfolds.
The parallel transport along $\gamma$ inside the local system 
$R^2\upsilon_*\ZZ$ coincides with $g$ by construction (cf.\ Remark~\ref{rmk:Theta and b}).
\end{proof}
%\begin{rmk}
%After this result was proved, Giulia Sacc\`a extended the intermediate Jacobian fibration $J_V$ to every smooth cubic fourfold $V$ (see \cite[Theorem~1.6]{Sacca:BirationalGeometryLSV}). Therefore, in the proof of Proposition~\ref{thm:my monodromy}, the argument invoking Lemma~\ref{lemma:U'} is not needed anymore. Nevertheless, we do believe that this stronger statement is still interesting and meaningful, and we decided to keep it.
%\end{rmk}
In the proof of Proposition~\ref{thm:my monodromy} we used the following result, which is well-known to experts but of which we could not find a reference.
\begin{lemma}\label{lemma:U'}
Let $M$ be a connected and complete Riemannian manifold. Let $\{D_k\}_{k\in I}$ be a countable set 
of closed submanifolds in M of (real) codimension strictly greater than $1$. 
Let $M'=M\setminus\bigcup_{k\in I}D_k$ and let $i:M'\to M$ be the inclusion. Then the induced map
\begin{equation*}
i_*:\pi_1(M',p)\longrightarrow\pi_1(M,p)
\end{equation*}
is surjective for every $p\in M'$.
\end{lemma}
\begin{proof}
Let $\gamma\in\pi_1(M,p)$ be a homotopy class.
Let $\cL_\gamma$ denote the set of loops $\delta$ in $M$ based at $p\in M$ such that 
$[\delta]=\gamma\in\pi_1(M,p)$. When endowed with the Hausdorff distance (induced by the complete 
metric on $M$), $\cL_\gamma$ becomes a complete metric space. For a closed submanifold 
$D\subset M$, let $\cL_\gamma(D)$ be the open subset of $\cL_\gamma$ consisting of loops disjoint 
from $D$. If the codimension of $D$ is strictly greater than $1$, Sard's theorem, applied to the inclusion 
map $\cL_\gamma\setminus\cL_\gamma(D)\to\cL_\gamma$, implies that $\cL_\gamma(D)$ is dense in 
$\cL_\gamma$. By Baire's Category theorem it follows then that $\cap_{k\in I}\cL(D_k)$ is dense in 
$\cL$ and hence there exists a loop $\bar{\delta}$ in $M$ which is disjoint from all the $D_{k}s$, i.e.\ 
$[\bar{\delta}]\in\pi_1(M',p)$. By construction $i_*([\bar{\delta}])=\gamma$.
\end{proof}

%%%%%%%%%%%%%%%%%%%%%%%%%%%%%%%%%%%%%%%%%
\subsection{Bridge to the O'Grady moduli space}\label{section:bridge}
We want to explain how to transport the monodromy operators arising from the LSV family
to the O'Grady family. We state now the main result and dedicate the rest of the section
to its proof.

Let $S$ be a generic $K3$ surface of genus $2$, that is $\Pic(S)=\ZZ H$ with $H^2=2$.
Define the classes 
$e=H-\widetilde{B}-\widetilde{\Sigma}$ and $f=H-2\widetilde{B}-\widetilde{\Sigma}$: they are the standard basis for a hyperbolic plane.
\begin{thm}\label{thm:LSV to OG}
Let $g\in\Ort^+(\oH^2(\widetilde{M}_S,\ZZ))$ be such that $g(e)=e$ and $g(f)=f$.
Then $g$ is a monodromy operator.
\end{thm}
The strategy is to degenerate the cubic fourfold to the chordal cubic fourfold: the intermediate
Jacobian fibration degenerates then to a (desingularised) moduli space of sheaves on the $K3$ 
surface of degree $2$, associated to the chordal cubic, with Mukai vector $(0,2H,-4)$. 
Such a moduli space is birational to the O'Grady moduli space $\widetilde{M}_S$. 
This approach is the one developed and used by Koll\'ar, Laza, Sacc\`a and Voisin 
(\cite{KollarLazaSaccaVoisin:Degenerations}) to prove that the intermediate Jacobian fibration
is of OG10 type. We study the induced map on the Picard lattices, in order to have an explicit
parallel transport operator to move the monodromy operators from 
$\cJ_V$ to $\widetilde{M}_S$.

%%%%%%%%%%%%%%%%%%%%%%%%%%%%%%%%%%%%%%%%%
\subsubsection{The degeneration}\label{subsection:degeneration}
Let $V_0$ be a generic chordal cubic fourfold, that is $V_0$ is the secant variety of the Veronese surface $P$ in $\PP^5$. 
Recall that $V_0$ is singular along $P$ and smooth elsewhere. Its S-equivalence class defines
a closed point in the boundary of the GIT-semistable compactification of the moduli space of
cubic fourfolds.
Now, let us pick a simple degeneration of a smooth cubic fourfold $V$ to $V_0$, that is
a pencil $\cV=\{F+tG=0\}_{t\in\Delta}$, where $V_0=\{F=0\}$ and $V=\{G=0\}$.
Consider the intersection 
$D=V\cap P$. 
Since $D\subset P\cong\PP^2$ is a smooth sextic curve, the double cover $f\colon S\to P$ 
ramified along $D$ is a smooth $K3$ surface. Moreover, since $V_0$ is general, $\Pic(S)=\ZZ H$,
where $H$ is a polarisation such that $H^2=2$.

Consider a general linear section $Y_0$ of $V_0$. This is a chordal threefold, i.e.\ the secant
variety of a rational quartic curve $\Gamma$ in $\PP^4\subset\PP^5$ ($\Gamma$ is the image
of the degree $4$ Veronese embedding of $\PP^1$). If $Y$ is the corresponding section
of $V$, and if the section is general enough, then $Y$ is a smooth cubic threefold and the 
intersection $Y\cap\Gamma\subset D$ consists of $12$ distinct points. 
The double cover $f|_C\colon C\to\Gamma$ ramified along these points is a smooth hyperelliptic
curve of genus $5$. As explained in \cite{Collino:TheFundamentalGroup}, 
via the degeneration defined by intersecting a generic linear section of $V$ as above with the 
degeneration $\cV$, the intermediate Jacobian $J_Y$ degenerates to the Jacobian $J_C$. This is a degeneration of principally polarised abelian varieties.

As explained in \cite[Section~5]{KollarLazaSaccaVoisin:Degenerations}, this implies that the central fibre $\cJ_{\cV_0}$
of the intermediate Jacobian fibration degeneration associated to $\cV$ has a reduced and
irreducbile component that is locally isomorphic to a fibration in Jacobians of hyperelliptic curves
of genus $5$.
In other words, the intermediate Jacobian fibration degeneration is birational to the desingularised
moduli space $\widetilde{M}_{(0,2H,-4)}(S)$. 

More precisely, if $U\subset\PP\oH^2(V,\cO_V(1))^*$ is the open subset of smooth linear sections, with an abuse of notation we also denote by $U$ the open subset in $|2H|$ parametrising smooth curves of genus $5$ on the $K3$ surface $S$ associated to the degeneration. We then denote by $\cJ_U$ the intermediate Jacobian fibration of $V$ and by $Pic^0_U$ the relative Picard variety of $S$. It is easy to see that $Pic^0_U\subset M_v(S)$, where $v=(0,2H,-4)$. Moreover $M_v(S)$ has a natural morphism to $|2H|$ given by assigning to each sheaf its Fitting support (see Example~\ref{exe:torsion sheaves}), and $M_v(S)$ is a compactification of the Picard variety fibration. The aforementioned results of Collino (\cite{Collino:TheFundamentalGroup}) can be re-phrased by saying that $\cJ_U$ degenerates to $Pic^0_U$. Koll\'ar, Laza, Sacc\`a and Voisin's result (\cite{KollarLazaSaccaVoisin:Degenerations}) says instead that the compactification $\cJ_V$ degenerates to the desingularised moduli space $\widetilde{M}_v(S)$.

We use this remark to compute the image of the classes $b_V$ and $\Theta_V$ under the degeneration, according to the following strategy.
\begin{itemize}
\item Recall that $b_V=\pi_V^*\cO_{\PP^5}(1)$ (see Theorem~\ref{thm:LSV} for the notation) is the class of the fibration. The Lagrangian fibration structure on $\widetilde{M}_v(S)$ is given by composing the desingularisation map $\pi\colon\widetilde{M}_v(S)\to M_v(S)$ with the Fitting support map $p\colon M_v(S)\to|2H|$. The corresponding class of the fibration is then $\tilde{b}_S=(p\circ\pi)^*\cO_{\PP^5}(1)$. Since the degeneration of $\cJ_V$ to $\widetilde{M}_v(S)$ preserves the Lagrangian fibration structure, we must have that $b_V$ goes to the class $\tilde{b}_S$. For later use, we put $b_S=p^*\cO_{\PP^5}(1)$, so that $\tilde{b}_S=\pi^*b_S$.
\item Recall that $\Theta_V$ is the closure of the image of the difference map (\ref{eqn:difference map})
$$f_V\colon\cF_U\times_U\cF_U\to\cJ_U\subset\cJ_V,$$
where $U\subset\PP\oH^0(V,\cO_V(1))^*$ is the open set of smooth linear
sections as before. Let us focus on a general fibre. 
In this case the Fano surface of lines of a smooth cubic threefold degenerates to the surface
$C^{(2)}\cup F_C$ (\cite[Proposition~2.1]{Collino:TheFundamentalGroup}). Here $C$ is the hyperelliptic
curve of genus $5$, $C^{(2)}$ is the symmetric product, $F_C\cong\PP^2$ and $C^{(2)}\cap F_C\cong K\cong\PP^1$.
The inclusion $K\subset F_C$ realises $K$ as a conic in $\PP^2$.
Notice that $\Alb(C^{(2)}\cup F_C)\cong\Alb(C^{(2)})\cong J_C$, via the restriction map
$\oH^1(C^{(2)}\cup F_C,\ZZ)\to\oH^1(C^{(2)},\ZZ)$ induced by the Mayer--Vietoris
sequence. We then look at the difference map 
$C^{(2)}\times C^{(2)}\to J_C$,
and in particular at its relative version
\begin{equation}\label{eqn:Collino Fano}
f_S\colon\cC_U^{(2)}\times_U\cC_U^{(2)}\longrightarrow\cJ_{\cC_U}\subset M_v(S).
\end{equation}
(According to the abuse of notation assumed before, $U$ is now the open subset of smooth curves in the linear system $|2H|$.)
Let $T_U$ be the image of $f_S$ and denote by $T_S$ its closure in $M_v(S)$.
Then $\Theta_V$ degenerates to the strict transform $\widetilde{T}_S$ of $T_S$.
\end{itemize}
The next task is then to write down the classes $\tilde{b}_S$ and $\widetilde{T}_S$ in a fixed basis of $\Pic(\widetilde{M}_v(S))\cong\Gamma_v^{1,1}$ (see Theorem~\ref{thm:factoriality}). 
Since $\tilde{b}_S=\pi^*b_S$, in this case it is enough to do it in the group $\Pic(M_v(S))\cong (v^\perp)^{1,1}$ (see item $(3)$ of Theorem~\ref{thm:big OG10}). For the class $\widetilde{T}_S$ instead, our first remark is that $T_S$ is a Cartier divisor (we will see shortly that $M_v(S)$ is locally factorial), so that we can consider its class in $\Pic(M_v(S))\cong (v^\perp)^{1,1}$, and then we will compute the multiplicity $m$ with which it contains the singular locus of $M_v(S)$. If $\widetilde{\Sigma}$ is the exceptional divisor of the desingularisation, then $\widetilde{T}_S=\pi^*T_S-m\widetilde{\Sigma}$. Notice that we systematically abuse notation by identifying a divisor with its cohomology class.
\medskip

We start by describing the basis of $\Pic(\widetilde{M}_v(S))$ we consider. Recall that $v=(0,2H,-4)$.
First of all, we notice that $M_v(S)$ is locally factorial,
so any Weil divisor is Cartier. In fact in this case $w=(0,H,-2)=v/2$ is such that $(w,u)\in2\ZZ$ for any $u\in\widetilde{\oH}^{1,1}(S,\ZZ)$, so that the claim follows from Remark~\ref{rmk:factoriality}. (Here we are using that $\Pic(S)=\ZZ H$.)
Recall from item $(3)$ of Theorem~\ref{thm:big OG10} that there is an isometry $\xi\colon(v^\perp)^{1,1}\to\Pic(M_v(S))$, and fix the basis $(v^\perp)^{1,1}=\langle a,b\rangle$, where $a=(-1,H,0)$ and $b=(0,0,1)$. If $D$ is a divisor on $M_v(S)$, then we write $D=\lambda \xi(a)+\mu \xi(b)$. 
%Now, since $M_v(S)$ is locally factorial, by Theorem~\ref{thm:factoriality} and the proof of \cite[Proposition~4.1]{PeregoRapagnetta:Factoriality}, it follows that there is a Hodge isometry  
%$$\tilde{\xi}\colon\Gamma_v^{1,1}=\langle a, b,\sigma\rangle\to\Pic(\widetilde{M}_v(S)),$$ where $\sigma$ is the degree $-6$ class in $\Gamma_v$ corresponding to the excceptional divisor $\widetilde{\Sigma}$.
The coefficients $\lambda$ and $\mu$ are determined by computing the intersection numbers $D.l_i$, $\xi(a).l_i$ and $\xi(b).l_i$, where $l_i\subset M_v(S)$, $i=1,2$, are two (linearly independent) curves classes.
If $D$ contains the singular locus $\Sigma_v$ of $M_v(S)$ with multiplicity $m$, then the class of the strict transform $\widetilde{D}$ of $D$ is
$$\widetilde{D}=\lambda\pi^*\xi(a)+\mu \pi^*\xi(b)-m\widetilde{\Sigma},$$
where $\widetilde{\Sigma}$ is the exceptional divisor of the desingularisation map.

%we determine its coefficients in $(v^\perp)^{1,1}$, with respect to the basis $v^\perp=\langle a,b\rangle$, where $a=(-1,H,0)$ and $b=(0,0,1)$, by computing the intersection numbers $D.l_i$, $a.l_i$ and $b.l_i$. Here $l_i\subset M_v(S)$, $i=1,2$, are two (linearly independent) curves. Let us first define the two curves.
Before continuing we make the following warning.
\begin{abusenotation}
If $M_v(S)$ is a moduli space and $\widetilde{M}_v(S)$ is its symplectic desingularisation, then by item $(3)$ of Theorem~\ref{thm:big OG10}, it follows that the pullback morphism
$\pi^*\colon\oH^2(M_v(S),\ZZ)\to\oH^2(\widetilde{M}_v(S),\ZZ)$
is Hodge isometric and injective, and moreover there is an isometry $\xi\colon v^\perp\to\oH^2(M_v(S),\ZZ)$ that preserves the Hodge structure.
We make the following abuses of notation:
\begin{itemize}
\item if $x\in v^\perp$, we write $x\in\oH^2(M_v(S),\ZZ)$ instead of $\xi(x)$;
\item if $\alpha\in\oH^2(M_v(S),\ZZ)$, we write $\alpha\in v^\perp$ instead of $\xi^{-1}(\alpha)$;
\item if $\alpha\in\oH^2(M_v(S),\ZZ)$, we write $\alpha\in\oH^2(\widetilde{M}_v(S),\ZZ)$ instead of $\pi^*(\alpha)$;
\item if $x\in v^\perp$, we write $x\in\oH^2(\widetilde{M}_v(S),\ZZ)$ instead of $\pi^*(\xi(x))$.
\end{itemize}
This is justified by the fact that both $\xi$ and $\pi^*$ preserve the Hodge and the lattice structures, so there is no ambiguity on where the lattice-theoretic operations between the classes (of any Hodge type) happen.
\end{abusenotation}
\begin{rmk}\label{rmk:modular curves}
Here we describe the standard method to define curves in a moduli space $M_v(S)$ (only for now $v$ is any Mukai vector). Let $\cE$ be a sheaf on the product $S\times C$, where $C$ is a curve, and assume that $\cE$ is flat over $C$. We say that $\cE$ is a family of sheaves on $S$ parametrised by $C$. Suppose moreover that for every $p\in C$ the restriction $\cE_p=\cE|_{S\times p}$ is semistable and $v(\cE_p)=v$. Then there exists a morphism
$$\varphi_{C,\cE}\colon C\longrightarrow M_v(S)$$
given by mapping $p\in C$ to the S-equivalence class $[\cE_p]$. In fact, since $M_v(S)$ co-represents a moduli functor, the morphism $\varphi_{C,\cE}$ is well defined by universal property.

If the image of $\varphi_{C,\cE}$ is $1$-dimensional, we will refer to the curve $\varphi_{C,\cE}(C)\subset M_v(S)$ as the curve associated to the pair $(C,\cE)$.
\end{rmk}

\emph{Vertical curve.}
Recall that $M_v(S)$ is fibred over $|2H|$ with general fibre $J_C$, the Jacobian of 
a smooth genus $5$ curve $C$. The fibration $M_v(S)\to|2H|$ is given by mapping any sheaf to its Fitting support. If we fix a point $p_0\in C$, then $C$ embeds in $J_C$ via the Abel--Jacobi map $AJ(p)=\cO_C(p-p_0)$. Define the sheaf
$$ \cE_{C,p_0}=(i\times id)_*\cO_{C\times C}\left(\Delta-p_0\times C\right),$$
where $i\colon C\to S$ is the natural embedding, $i\times id\colon C\times C\to S\times C$ and
$\Delta\subset C\times C$ is the diagonal. The sheaf $\cE_{C,p_0}$ is a sheaf on $S\times C$, flat over $C$. Moreover, for every $p\in C$, the restriction $\cE_{C,p_0}|_{S\times p}=i_*\cO_C(p-p_0)$ is a stable sheaf in $M_v(S)$. The curve $l_1$ is then the curve associated to the pair $(C,\cE_{C,p_0})$ (cf.\ Remark~\ref{rmk:modular curves}).
\begin{rmk}
This curve is not canonical: it depends on the fixed point, so one should really write $\cE_{C,p_0}$.
We drop the reference to $p_0$, and only write $\cE_C$, because it will not be important in the computations.
\end{rmk}

\emph{Horizontal curve.}
We start with the following remark.
\begin{lemma}\label{lemma:pencil}
There exists a pencil $L\cong\PP^1\subset|2H|\cong\PP^5$ and three points $p_1,p_2,p_3\in L$ such that, letting $\cC\subset S\times L$ be the family of curves on $S$ parametrised by $L$ and $\cC_p$ the curve corresponding to $L$, the following holds:
\begin{itemize}
\item if $p\neq p_1,p_2,p_3$, then $\cC_p$ is irreducible and for $p$ general (in a sense made precise in the proof) it is smooth;
\item if $p=p_i$ for $i=1,2,3$, then $\cC_{p_i}=\cC_{i,1}+\cC_{i,2}$ where $C_{i,j}\in|H|$ is a smooth curve of genus $2$;
\item the base locus of $L$ consists of eight points $q_1,\cdots,q_8$.
\end{itemize}
\end{lemma}
\begin{proof}
Recall that $f\colon S\to\PP^2$ is a double cover ramified along a sextic curve. Then
$H=f^*\cO(1)$ and $|2H|\cong|\cO(2)|$. The latter linear system parametrises conics
in $\PP^2$ and a Lefschetz pencil $L\subset|\cO(2)|$ has only three critical points
(corresponding to two incident lines - these are the points corresponding to $p_1$, $p_2$ and $p_3$). Let us denote by $\cC'$ the family of conics parametrised by $L$. If $L$ is general enough, then it has four
base points. The pullback $\cC=f^*\cC'$ is a family of curves in $S$ of class $|2H|$, parametrised by $L$ and satisfying the three properties in the lemma. In particular the curves $\cC_p$, for $p\neq p_1,p_2,p_3$, are always irreducible, and smooth when the corresponding conic is not tangent to the sextic curve.
\end{proof}
Denote by $j\colon\cC\to S\times L$ the natural inclusion and define
$$\cE_L=j_*\cO_{\cC}.$$
$\cE_L$ is a flat family of sheaves on $S$ parametrised by $L$. Moreover a direct computation shows that the sheaf $i_{p*}\cO_{\cC_p}$ is stable for every $p\in L$ (here $i_p\colon\cC_p\to S$ is the natural embedding). 
Using the convention set in Remark~\ref{rmk:modular curves}, we denote by $l_2$ the curve associated to the pair $(L,\cE_L)$. 
\begin{rmk}\label{rmk:zero section}
Notice that $l_2$ is the closure in $M_v(S)$ of a line in the zero section on $\cJ_U$ and it does not meet the singular locus of $M_v(S)$.
\end{rmk}

\begin{lemma}\label{lemma:some intersection}
The following intersection products hold:
$$ a.l_1=5, \qquad a.l_2=-1, \qquad b.l_1=0,\qquad b.l_2=1.$$
\end{lemma}
\begin{proof}
Let us outline the computation of $a.l_1$.
By \cite[Theorem~8.1.5]{HuybrechtsLehn:ModuliSpaces}, if $G$ is a sheaf on $S$ such that 
the Mukai vector %\footnote{The pairing on $\Kt(S)$ used in \cite[Theorem~8.1.5]{HuybrechtsLehn:ModuliSpaces} is $\chi(F\otimes G)$, while the Mukai pairing is 
%$\chi(F\otimes G^\vee)$. This is why we need to choose $G$ such that $\mathfrak{v}(G^\vee)=a$.}
of $G^\vee$ is $a$, then 
\begin{equation}\label{eqn:8.1.5}
a.l_1=\deg(\phi_{C,\cE_C}^*a)=\deg c_1\left(\pi_{C!}\left(\cE_C\otimes\pi_S^*G\right)\right),
\end{equation}
where $\varphi_{C,\cE_C}\colon l_1\to M_v(S)$ is the classifying morphism (cf.\ Remark~\ref{rmk:modular curves}).
Using the Grothendieck--Riemann--Roch formula, we get
$$a.l_1=\pi_{C*}\left[\ch(\cE_C)\ch(\pi_S^*G)\td_{\pi_C}\right]_3=\pi_{C*}\left[\ch(\cE_C)\pi_S^*(a^\vee\sqrt{\td_S})\right]_3,$$
where the last equality follows from the fact that $\td_{\pi_C}=\pi_S^*\td_S$. 
Since $a=(-1,H,0)$, then $a^\vee=(-1,-H,0)$ and so
$$ \pi_S^*(a^\vee\sqrt{\td_S})=(-1,-[H\times C],-[p_0\times C],0),$$
where the square brackets indicate the class in cohomology of the corresponding cycle.
To compute $\ch(\cE_C)$, we use the Grothendieck--Riemann--Roch formula again:
$$\ch(\cE_C)=(i\times\id)_*\left(\ch\left(\cO_{C\times C}\left(\Delta-p_0\times C\right)\right)\td_{i\times\id}\right).$$
Now, by a direct computation,
$$\ch(\cO_{C\times C}(\Delta-p_0\times C))=(1,[\Delta]-[p_0\times C],-5)$$
and
$$ \td_{i\times\id}=\left(1,-\frac{1}{2}[K_C\times C],0\right),$$
where $K_C$ is the canonical divisor of $C$.
Putting all together we eventually get
$$ \ch(\cE_C)=\left(0,[C\times C],[(i\times\id)(\Delta)]-[p_0\times C]-\frac{1}{2}[K_C\times C],-9\right)$$
and hence 
$$ a.l_1 =\pi_{C*}\left[\ch(\cE_C)\pi_S^*(a^\vee\sqrt{\td_S})\right]_3= $$
$$ =-[C\times C].[p_0\times C]-[H\times C].([(i\times\id)(\Delta)]-[p_0\times C]-\frac{1}{2}[K_C\times C])+9=$$
$$ =0-4+9=5.$$
The computation of $b.l_1$ is similar and we leave it to the reader. 

The intersections $a.l_2$ and $b.l_2$ are also computed with the same technique. We only recall here the computation of $\ch(\cE_L)$, and leave to the reader the conclusion of the proof.

First of all, recall that $\cE_L=j_*\cO_{\cC}$, where $j\colon\cC\to S\times L$ is the natural inclusion. Notice that $\cC\cong \operatorname{Bl}_{q_1,\cdots,q_8}(S)$, where $q_1,\cdots,q_8$ are the base points of the pencil $L$ (cf.\ Lemma~\ref{lemma:pencil}). Then by the Grothendieck--Riemann--Roch formula we have
$$\ch(\cE_L)=j_*(\operatorname{td}_j),$$
and from the normal bundle sequence we get
$$ \operatorname{td}_j=\td(N_{\cC/S\times L})^{-1}=\left(1,-\frac{1}{2}c_1(N_{\cC/S\times L}),\frac{1}{6}c_1(N_{\cC/S\times L})^2\right).$$
%where $N_{\cC/S\times L}$ is the normal bundle of $\cC$ in $S\times L$. 
Therefore one needs to compute $c_1(N_{\cC/S\times L})$. Using the adjunction formula on $\cC\subset S\times L$ and the isomorphism $\cC\cong\operatorname{Bl}_{q_1,\cdots,q_8}(S)$, we get
$$c_1(N_{\cC/S\times L})=K_{\cC}-K_{S\times L}|_{\cC}=E_1+\cdots+E_8-\pi_L^*K_L.$$
Here $E_i$ is the exceptional divisor in $\cC$ corresponding to the point $q_i$ and $\pi_L$ is the restriction to $\cC$ of the natural projection $S\times L\to L$.

Putting everything together we eventually have
$$\ch(\cE_L)=\left(0,[\cC],-\frac{E_1+\cdots+E_8-\pi_L^*K_L}{2},4\right).$$
\end{proof}

\begin{prop}\label{prop:b_V}
The class $b_V$ degenerates to the class $\tilde{b}_S=b$. 
\end{prop}
\begin{proof}
We have already remarked that $b_V$ degenerates to $\tilde{b}_S=\pi^*b_S$, so the claim follows once we prove that $b_S=b\in v^\perp$.

Since $l_1$ is concentrated in a fibre, by the projection formula $b_S.l_1=0$.
On the other hand, since $l_2$ is a line inside the zero section, we easily get $b_S.l_2=1$.
The claim follows.
\end{proof}

According to the strategy outlined before, in order to compute the class of $\widetilde{T}_S$, we first compute the class of $T_S$ in the given basis of $(v^\perp)^{1,1}$.
\begin{prop}\label{prop:Theta}
$T_S=a-2b$.
\end{prop} 
\begin{proof}
Since $T_S$ is a theta divisor, by Poincar\'e formula it follows that 
$T_S.l_1=5$. 
By taking $d$ to be the intersection number $d=T_S.l_2$ and using 
Lemma~\ref{lemma:some intersection}, we see that $T_S=a+(d+1)b$.
Finally, it follows from Lemma~\ref{lemma:theta exceptional} that $-2=T^2_S=2(d+2)$, so that $d=-3$ and the claim follows.
\end{proof}
The final step consists in computing the multiplicity of the singular locus $\Sigma_v$ of $M_v(S)$ in $T_S$.
\begin{prop}\label{prop:multiplicity T_S}
$\Sigma_v$ is not contained in $T_S$. 
\end{prop}
The proof of the proposition relies on the following remark.
\begin{lemma}\label{lemma:0=4}
Let $w=(0,2H,0)$. Then $M_v(S)$ is isomorphic to $M_w(S)$ and the isomorphism preserves the singular locus.

Moreover, the theta divisor $T_S$ is isomorphic to the closure in $M_w(S)$ of the relative Brill--Noether divisor $\mathcal{W}_4^0$ of the family of smooth curves in $|2H|$.
\end{lemma}
\begin{proof}
By Lemma~\ref{lemma:tensor H}, tensoring by $H$ gives an isomorphism $\phi\colon M_v(S)\to M_w(S)$ that preserves the singular locus. More precisely, any smooth curve $C$ in $|2H|$ is hyperelliptic and has a unique $g^1_2$. By definition, the restriction of $H$ to $C$ coincides with $2g^1_2$, hence the morphism $\phi$ acts fibrewise by sending a line bundle $L$ on a smooth curve $C$ to the line bundle $L\otimes2g^1_2$ on $C$.

Now, the moduli space $M_w(S)$ contains the Picard variety fibration $Pic^4_{U}$, where $U\subset|2H|$ is the locus of smooth curves. If we denote by $\cC_U$ the family of curves parametrised by $U$, the Brill--Noether variety $\mathcal{W}_4^0\subset Pic^4_{U}$ coincides with the image of the morphism
$$ g_S\colon\cC_U^{(4)}\longrightarrow Pic^4_{U},$$
given by $g_S((p_1,p_2,p_3,p_4);C)=i_*\cO_C(p_1+p_2+p_3+p_4)$, where $i\colon C\to S$ is the embedding.
Let us call $E_U$ the image of $g_S$. We claim that $E_U=\phi(T_U)$ (recall that $T_U$ was defined as the image of the map (\ref{eqn:Collino Fano})).

Let $C\in U$ be a smooth curve in $|2H|$ and consider the line bundle $\cO_C(p_1+p_2-q_1-q_2)$. Any divisor in the equivalence class of $H$ cuts $C$ in four points (i.e.\ a pair of $g_1^2$), and up to change the linear equivalence representative, we can always suppose that two of these points are $q_1$ and $q_2$. It follows that $i_*\cO_C(p_1+p_2-q_1-q_2)\otimes H=i_*\cO_C(p_1+p_2+r_1+r_2)$, where $r_1$ and $r_2$ are the two residue points of the intersection of $C$ with $H$.
\end{proof}
\begin{proof}[Proof of Proposition~\ref{prop:multiplicity T_S}]
By Lemma~\ref{lemma:0=4} above, it is enough to prove that $\Sigma_w$ is not contained in the closure $\overline{\mathcal{W}}_4^0$ of $\mathcal{W}_4^0$; in particular, it is enough to prove that a general point of $\Sigma_w$ does not belong to $\overline{\mathcal{W}}_4^0$. Recall from \cite{LehnSorger:Singularity} (see also \cite{O'Grady:10dimPublished}) that a general point $x\in\Sigma_w$ is an S-equivalence class $[F_1\oplus F_2]$, where $F_j=i_{j*}L_j$, $L_j\in\Pic^{1}(C_j)$ is general and $C_j\in|H|$ is a smooth curve. The support of $x$ is the curve $C_0=C_1\cup C_2\in|2H|$. If $p\colon M_w(S)\to|2H|$ is the map that associates to any sheaf its Fitting support, then we want to show that $x$ does not belong to the intersection $\overline{\mathcal{W}}_4^0\cap p^{-1}(C_0)$. 

First of all, let us describe the fibre $\overline{\mathcal{W}}_4^0\cap p^{-1}(C_0)$. By construction, the fibre $p^{-1}(C_0)$ coincides with the generalised Jacobian $\bar{J}_{C_0}$, as constructed in \cite{Caporaso:Compactification} (see also item $(4)$ of \cite[Lemma~2.3]{Rapagnetta:BeauvilleFormIHSM}). So the fibre $\overline{\mathcal{W}}_4^0\cap p^{-1}(C_0)$ can be understood classically as the locus of sheaves with sections, i.e.\ $F\in\overline{\mathcal{W}}_4^0\cap p^{-1}(C_0)$ if and only if $h^0(F)>0$ (\cite[Theorem~5.3]{Alexeev:CompactifiedJacobians}).

Now, in any family of semistable sheaves, the locus of sheaves $F$ such that $h^0(F)>0$ is closed on the base; therefore the locus in $M_w(S)$ of S-equivalence classes where at least one representative has sections is still closed. Thanks to this remark, we reduce to prove that any sheaf in the S-equivalence class of a general point $x\in\Sigma_w(S)$ has no sections. 

%Let $F\in M_w(S)$ be a strictly semistable sheaf supported on $C_0$; then $F=i_*G$, where $G$ is a torsion free sheaf on $C_0$. If we call $G_1$ and $G_2$ the torsion free parts of the restrictions of $G$ to $C_1$ and $C_2$, respectively, we have a short exact sequence
%\begin{equation*}
%0\longrightarrow i_*G\longrightarrow i_{1*}G_1\oplus i_{2*}G_2\longrightarrow Q\longrightarrow0,
%\end{equation*}
%where $Q$ is the structure sheaf of the intersection $C_1\cap C_2=\{n_1,n_2\}$. In this situation, up to renaming the indeces, the S-equivalence class of $F$ is $[i_{1*}G_1(-n_1-n_2)\oplus i_{2*}G_2]$. 

Let $x=[i_{1*}L_1\oplus i_{2*}L_2]$ be a general point in $\Sigma_w$. Then $\oH^0(L_1)=0=\oH^0(L_2)$: in fact $L_j\in\Pic^1(C_j)$ is general, hence not effective, since $C_j$ has genus $2$. Strictly semistable sheaves $F$ such that $[F]=x$ are described (up to isomorphism) as kernels of the surjections (cf.\ proof of \cite[Lemma~2.3]{Rapagnetta:BeauvilleFormIHSM})
$$ i_{1*}L_1(n_1+n_2)\oplus i_{2*}L_2\longrightarrow\!\!\!\!\!\rightarrow Q.$$
If $F$ is any such a kernel, computing the long exact sequence in cohomology we get that $\oH^0(F)$ is described as the sub-vector space of $\oH^0(L_1(n_1+n_2))\oplus\oH^0(L_2)\cong\oH^0(L_1(n_1+n_2))$ of sections vanishing at the points $n_1$ and $n_2$ of $C_1$. This sub-vector space coincides with the cohomology group $\oH^0(L_1)$, which is zero by assumption. It follows that $\oH^0(F)=0$, therefore no representatives of $x$ belong to $\overline{\mathcal{W}}_4^0\cap p^{-1}(C_0)$ and the claim is proved.
\end{proof}
\begin{cor}
The theta divisor $\Theta_V$ degenerates to $\widetilde{T}_S=\pi^*T_S=a-2b$.
\end{cor}
\begin{proof}
We have already remarked that $\Theta_V$ degenerates to $\widetilde{T}_S$. 
By Proposition~\ref{prop:multiplicity T_S}, the multiplicity of $\Sigma_v$ in $T_S$ is $0$, so that $\widetilde{T}_S=\pi^*T_S$ and the claim follows from Proposition~\ref{prop:Theta}.
\end{proof}

%%%%%%%%%%%%%%%%%%%%%%%%%%%%%%%%%%%%%%%%
\subsubsection{The hyperelliptic birational map}\label{subsection:hyperelliptic}
Let $U\subset|2H|$ be the locus of smooth curves and $Pic^2_U$ the Picard variety fibration over $U$. Since all the curves parametrised by $U$ are hyperelliptic there exists a canonical isomorphism
$Pic^0_U\cong Pic^2_U$ given fibrewise by tensoring with the unique $g^1_2$ on each fibre.
Recalling that $Pic^0_U\subset M_{(0,2H,-4)}$ and that $Pic^2_U\subset M_{(0,2H,-2)}$ (see Example~\ref{exe:torsion sheaves}), this also gives a canonical birational map
\begin{equation*}
\psi\colon M_{(0,2H,-2)}\dashrightarrow M_{(0,2H,-4)}.
\end{equation*}
We denote by $\widetilde{\psi}$ the induced birational maps on the symplectic desingularisations
of these spaces.
Only for the rest of this section, we use the notation $v_2=(0,2H,-2)$ and $v_0=(0,2H,-4)$.
Recall that $M_{v_0}(S)$ is locally factorial and so 
$$\Pic(\widetilde{M}_{v_0}(S))\cong(v_0^\perp)^{1,1}\oplus\ZZ\sigma_0=\langle a,b,\sigma_0\rangle.$$
On the other hand, $M_{v_2}(S)$ is $2$-factorial and 
$$\Pic(\widetilde{M}_{v_2}(S))=\Gamma_{v_2}^{1,1}=$$
$$
=\left\{\left(w_{m,n},k\frac{\sigma_2}{2}\right)\mid 
(w_{m,n},v_2^\perp)\subset\ZZ, k\in2\ZZ\Leftrightarrow m,n\in2\ZZ\right\},
$$
where $w_{m,n}=\left(-m,\frac{m}{2}H,\frac{n}{2}\right)$. Both these statements follow from Theorem~\ref{thm:factoriality} and the proof of \cite[Proposition~4.1]{PeregoRapagnetta:Factoriality}.
In the following we write $(v_2^\perp)^{1,1}=\langle a_2,b_2\rangle$,
where $a_2=(-2,H,0)$ and $b_2=(0,0,1)$. Finally, recall that $\sigma_0^2=-6=\sigma_2^2$.

Since $\widetilde{\psi}$ is a birational isomorphism between irreducible holomorphic symplectic manifolds, its pullback induces a Hodge isometry (\cite{Huybrechts:BasicResults})
$$\widetilde{\psi}^*\colon\oH^2(\widetilde{M}_{v_0}(S),\ZZ)\longrightarrow\oH^2(\widetilde{M}_{v_2}(S),\ZZ).$$
The goal of this section is to write down explicitly the induced isometry
$$\widetilde{\psi}^*\colon\Pic(\widetilde{M}_{v_0}(S))=\Gamma_{v_0}^{1,1}\to\Gamma_{v_2}^{1,1}=\Pic(\widetilde{M}_{v_2}(S))$$
with respect to the generators fixed above. 
This goal is achieved in Proposition~\ref{prop:psi}, but before that we need to investigate some features about the space $\widetilde{M}_{v_2}(S)$ and the map $\widetilde{\psi}$.
\medskip

Recall from Example~\ref{exe:torsion sheaves} that $\widetilde{M}_{v_2}(S)$ is also a Lagrangian fibration, this structure being given by the composition of the desingularisation map $\pi_2\colon\widetilde{M}_{v_2}(S)\to M_{v_2}(S)$ and the Fitting support map $p\colon M_{v_2}(S)\to|2H|$. The class of the fibration is $\tilde{b}_{v_2}=(p\circ\pi_2)^*\cO_{\PP^5}(1)$.
\begin{lemma}\label{lemma:b_2}
$\tilde{b}_{v_2}=b_2$.
\end{lemma}
\begin{proof}
The proof is the same as the proof of Proposition~\ref{prop:b_V}, using computations similar to those in Lemma~\ref{lemma:some intersection}, where instead of the curves $l_1$ and $l_2$ one uses the curves $l_1'$ and $l_2'$ that we now define.

The curve $l_1'$ is the image of the vertical curve $l_1$ via $\psi^{-1}$. In fact notice that $l_1$ is contained in $Pic^0_U\subset M_{v_0}(S)$, so the morphism $\psi^{-1}$ is well-defined on it. In particular, $l_1'=(C,\cE_C')$, where (we use the same notations used to define the curve $l_1$)
$$ \cE_{C}'=(i\times id)_*\cO_{C\times C}\left(\Delta-p_0\times C+K_C\times C\right).$$

To construct the curve $l_2'$ we use again the pencil $L\subset|2H|$ of Lemma~\ref{lemma:pencil}. Let $\cC$ be the family of curves in $S$ parametrised by $L$. Recall that $\cC$ is the pullback of a family $\cC'$ of conics in $\PP^2$. If $Z'\colon L\to\cC'$ is a section, then the pullback via the $2\colon1$ cover $S\to\PP^2$ is a bi-section $Z\colon L\to\cC$. For every $p\in L$ such that $\cC_p$ is smooth, $Z_p$ represents the $g^1_2$. We can choose $Z$ in such a way that $Z_p$ is composed by two distinct points (this is always true, after possibly a finite base change); in particular $Z_{p_i}$ does not pass through the node of $\cC_{p_i}$, for $i=1,2,3$. Then we define the family of sheaves $j_*\cO_{\cC}(Z)$, where $j\colon\cC\to S\times L$ is the embedding. This is a well defined family of semistable sheaves on $S$ parametrised by $L$ and $l_2'$ is the curve associated to the pair $(L,j_*\cO_{\cC}(Z))$ (cf.\ Remark~\ref{rmk:modular curves}).

Now, it is clear that $\tilde{b}_{v_2}.l_1'=0$ and $\tilde{b}_{v_2}.l_2'=1$. Let us compute $a_2.l_1'$, $b_2.l_1'$ and $b_2.l_2'$. These intersection numbers are computed as in Lemma~\ref{lemma:some intersection} and therefore we will only briefly comment on them.

The Chern character of $\cE_C'$ is as follows (use notations as in the proof of Lemma~\ref{lemma:some intersection})
$$ \ch(\cE_C')=\left(0,[C\times C],[(i\times\id)(\Delta)]-[p_0\times C]+\frac{1}{2}[K_C\times C],-4\right),$$
so that $a_2.l_1'=4$ and $b_2.l_1'=0$. It follows that $\tilde{b}_{v_2}$ is a multiple of $b_2$.

Finally, the intersection $b_2.l_2'$ is easier to compute. In fact, since $b_2=(0,0,1)$, we do not need to know the whole shape of $\ch(j_*\cO_{\cC}(Z))$, but only its first Chern class, which is equal to $[\cC]$. It follows that $b_2.l_2'=[p_0\times C].[\cC]=1$ and so $\tilde{b}_{v_2}=b_2$.
\end{proof}
\begin{rmk}\label{rmk:zero-two-section}
By a direct check, the sheaves $i_{p_{i*}}\cO_{\cC_{p_i}}(Z_{p_i})$ are strictly semistable, and the same holds if we choose another bi-section $Z$ such that $Z_p$ represents the $g^1_2$ of the smooth curve $\cC_p$. Moreover, by a direct computation it can be shown that the S-equivalence class of $\cO_{\cC_{p_i}}(Z_{p_i})$ is independent of the choice of $Z$, i.e.\ the morphism $\psi^{-1}$ is well-defined on the generic point of the fibre over a reduced and reducible curve.

As a consequence we get that $l_2'=\psi^{-1}(l_2)$ and that $l_2'\cap\Sigma_{v_2}=\{p_1,p_2,p_3\}$. In particular, the morphism $\psi^{-1}$ does not preserve the singular locus of $M_{v_0}(S)$.
\end{rmk}

\begin{prop}\label{prop:psi}
$\widetilde{\psi}^*\colon\Pic(\widetilde{M}_{v_0}(S))\to\Pic(\widetilde{M}_{v_2}(S))$
is the isometry 
$$\begin{array}{rcl}
a & \longmapsto & \frac{1}{2}a_2+\frac{3}{2}b_2-\frac{1}{2}\sigma_2\\%=\left(\left(-1,\frac{H}{2},\frac{3}{2}\right),-\frac{\sigma_2}{2}\right)\\
b & \longmapsto & b_2 \\
\sigma_0 & \longmapsto & 3b_2-\sigma_2.
\end{array}$$
\end{prop}
\begin{proof}
First, we remark that the morphism $\widetilde{\psi}$ is defined fibrewise, so that the class of the fibrations are preserved, i.e.\ $\widetilde{\psi}^*\tilde{b}_{v_0}=\tilde{b}_{v_2}$. Then, by Proposition~\ref{prop:b_V} and Lemma~\ref{lemma:b_2}, we get 
$$\widetilde{\psi}^*(b)=b_2.$$
Writing $\widetilde{\psi}^*(a)=(w_{m,n},\frac{k}{2}\sigma_2)$, from the relations 
$$\widetilde{\psi}^*(a)^2=2\qquad\mbox{ and }\qquad (\widetilde{\psi}^*(a),b_2)=(\widetilde{\psi}^*(a),\widetilde{\psi}^*(b))=1$$
it follows that 
$$ \widetilde{\psi}^*(a)=\frac{1}{2}a_2+\frac{n}{2}b_2+\frac{k}{2}\sigma_2=\left(\left(-1,\frac{1}{2}H,\frac{n}{2}\right),k\frac{\sigma_2}{2}\right)$$
where $n,k\in\ZZ$ are related by the equation
$$ n-\frac{3}{2}k^2-\frac{3}{2}=0.$$

Let us now consider the inverse isometry $(\widetilde{\psi}^*)^{-1}$, and write $(\widetilde{\psi}^*)^{-1}(\sigma_2)=xa+yb+z\sigma_0$. From the relation
$$
\left((\widetilde{\psi}^*)^{-1}(\sigma_2),b\right)=\left((\widetilde{\psi}^*)^{-1}(\sigma_2),(\widetilde{\psi}^*)^{-1}(b_2)\right)=0
$$
we get that $(\widetilde{\psi}^*)^{-1}(\sigma_2)=yb+z\sigma_0$, and from the relation $(\widetilde{\psi}^*)^{-1}(\sigma_2)^2=-6$ we get that $z=\pm1$. Applying $\widetilde{\psi}^*$ to this equation, we get 
\begin{equation}\label{eqn:boh}
\widetilde{\psi}^*(\sigma_0)=z\left(\sigma_2-yb_2\right).
\end{equation}

Moreover, 
$$ 0=(\widetilde{\psi}^*(a),\widetilde{\psi}^*(\sigma_0))=-zy-3zk,$$
which implies that $y=-3k$.

Notice that $y=3$. In fact, let us intersect (\ref{eqn:boh}) with $l_2'$: since $l_2'$ is
a translate of the zero section on $M_{v_0}(S)$ (see Remark~\ref{rmk:zero-two-section}) and does not intersect the singular locus there,
it follows that $\widetilde{\psi}^*(\sigma_0).l_2'=0$; on the other hand $b_2.l_2'=b_0.l_2=1$
(cf.\ Lemma~\ref{lemma:some intersection} and proof of Lemma~\ref{lemma:b_2}) and $\sigma_2.l_2'=3$ as it follows from 
Remark~\ref{rmk:zero-two-section} (in fact the last intersection happens in the smooth locus of 
$\Sigma_2$ and, by generality of $l_2'$, it is transversal).

Finally, let us show that $z=-1$. By Remark~\ref{rmk:zero-two-section} there exists a point $x\in\Sigma_{v_2}$ such that $\psi(x)\in M_{v_0}$ is smooth. Consider the line $\delta=\pi_2^{-1}(x)$, where $\pi_2\colon\widetilde{M}_{v_2}\to M_{v_2}$ is the desingularisation map. By a direct computation we have that $\sigma_2.\delta=-2$ and $b_2.\delta=0$. On the other hand, the intersection $\widetilde{\psi}^*(\sigma_0).\delta$ is transverse, since $\psi(x)$ is smooth, therefore it is positive. Intersecting the relation (\ref{eqn:boh}) with $\delta$ shows then that $z=-1$.
\end{proof}

%%%%%%%%%%%%%%%%%%%%%%%%%%%%%%%%%%%%%%%%%
\subsubsection{The O'Grady moduli space}\label{subsection:O'Grady}
Let $S$ be as usual a projective $K3$ surface with a polarisation $H$ of degree $2$ and let
$\D^b(S)$ be the derived category of coherent sheaves on $S$. Let $\Delta\subset S\times S$
be the diagonal and $F_{\Delta}\colon\D^b(S)\to\D^b(S)$ the Fourier--Mukai transform
with kernel the ideal sheaf $I_\Delta$.
As usual we denote by $M_S$ the O'Grady moduli space $M_{(2,0,-2)}(S)$.
\begin{prop}\label{prop:G}
The exact equivalence
$$\mathcal{G}\colon\D^b(S)\stackrel{-\otimes H}{\longrightarrow}\D^b(S)\stackrel{F_{\Delta}}{\longrightarrow}\D^b(S)\stackrel{-\otimes H^\vee}{\longrightarrow}\D^b(S)$$
induces a birational map 
$$\mathcal{G}\colon\widetilde{M}_{(0,2H,-2)}(S)\dashrightarrow\widetilde{M}_S.$$
Moreover, the induced map on the Picard lattices is
$$ \left(\left(-m,\frac{m}{2}H,\frac{n}{2}\right),k\frac{\sigma_2}{2}\right)\longmapsto
\frac{m+n}{2}H-n\widetilde{B}+\frac{k-n}{2}\widetilde{\Sigma}.$$
\end{prop}
\begin{rmk}
Notice that $m+n\in2\ZZ$ since $\left(\left(-m,\frac{m}{2}H,\frac{n}{2}\right),v_2^\perp\right)\subset\ZZ$ by definition.
\end{rmk}
The proposition follows from the following two lemmata and Example~\ref{example:O'Grady moduli space}.
\begin{lemma}[\protect{\cite[Lemma~2.26]{PeregoRapagnetta:DeformationsOfO'Grady}}]\label{lemma:tensor H}
Tensoring by a multiple $kH$ of the polarisation does not change the stability of a sheaf. 
In particular, if $w=(r,cH,s)$ and $w'=(r,(c+rk)H,s+2ck+rk^2)$, then the induced morphism 
$$ -\otimes(kH)\colon\widetilde{M}_{w}(S)\longrightarrow\widetilde{M}_{w'}(S)$$
is an isomorphism and sends the exceptional divisor of the former to the exceptional divisor
of the latter.
\end{lemma}
\begin{lemma}
The Fourier--Mukai transform $F_\Delta\colon\D^b(S)\to\D^b(S)$ induces
a birational morphism 
$F_\Delta\colon\widetilde{M}_{(0,2H,2)}(S)\to\widetilde{M}_{(2,2H,0)}(S)$,
$E\mapsto F_\Delta(E)^\vee$, such that 
the exceptional divisor of the former is sent to the exceptional divisor of the latter.

The induced map $\Pic(\widetilde{M}_{(0,2H,2)})\to\Pic(\widetilde{M}_{(2,2H,0)})$
on the Picard lattices is
$$
\left(\left(m,\frac{m}{2}H,\frac{n}{2}\right),k\frac{\sigma_2}{2}\right)\longmapsto
\left(\left(-\frac{n}{2},-\frac{m}{2}H,-m\right),k\frac{\sigma_2}{2}\right).
$$
\end{lemma}
\begin{proof}
This is essentially \cite[Proposition~4.1.2]{O'Grady:10dimPublished}. 
In fact, following the notation therein, let $\cJ^0$ be the open subset of 
$\widetilde{M}_{(0,2H,2)}(S)$ consisting of sheaves $E=i_*L$ where $i\colon C\to S$ is the 
inclusion, $C$ is smooth and $L$ is a globally generated line bundle with $h^0(L)=2$.
Notice that $\chi(L)=2$, so $h^1(L)=0$.
The short exact sequence $0\to I_\Delta\to\cO_{S\times S}\to\cO_\Delta\to0$, induces a 
short exact sequence of complexes
$$ 0\longrightarrow F_\Delta(i_*L)\longrightarrow F_{\cO_{S\times S}}(i_*L)\stackrel{f}{\longrightarrow} F_{\cO_{\Delta}}(i_*L)\longrightarrow0.$$
Since $i_*L\in\cJ^0$, we have that $F_{\cO_{S\times S}}(i_*L)=\oH^0(L)\otimes\cO_S$ and
$f\colon\oH^0(L)\otimes\cO_S\to i_*L$ is the evaluation map, which is surjective by 
hypothesis. 
The proof is now reduced to the proof of \cite[Proposition~4.1.2]{O'Grady:10dimPublished}.

The statement about the exceptional divisor follows from the same computation applied to the 
general point of the singular locus $\Sigma_v\cong\Sym^2 M_{v/2}(S)$.

Finally, the change of sign in the last statement follows from 
\cite[Proposition~2.5]{Yoshioka:ModuliSpacesAbelianSurfaces}
\end{proof}

%%%%%%%%%%%%%%%%%%%%%%%%%%%%%%%%%%%%%%%%%
\subsubsection{Proof of Theorem~\ref{thm:LSV to OG}}
Let $\cV\to\Delta$ be the degeneration to the chordal cubic fourfold considered in
\ref{subsection:degeneration}. The open subset $\Delta'=\Delta\setminus\{0\}$ maps
to the period domain $\Omega_{\operatorname{OG10}}$ of irreducible holomorphic symplectic 
manifolds of OG10 type via the periods of the associated Laza--Sacc\`a--Voisin family 
$\cJ_{\Delta'}$.
The main results of \cite{KollarLazaSaccaVoisin:Degenerations} say that the central fibre 
of the degeneration $\cJ_{\Delta}\to\Delta$ can be replaced by a smooth member that is birational to the moduli
space $\widetilde{M}_{(0,2H,-4)}(S)$, where $S$ is the (generic) $K3$ surface dual to the 
chordal cubic fourfold $\cV_0$. 
This means that the map $\Delta'\to\Omega_{\operatorname{OG10}}$ can be extended
to a map $\Delta\to\Omega_{\operatorname{OG10}}$ by sending $0$ to
the period of $\widetilde{M}_{(0,2H,-4)}(S)$. 
Finally, since the period map is surjective (\cite{Huybrechts:BasicResults}), one gets 
a family $p_1\colon\cX_1\to\Delta$ with two distinguished 
members corresponding to $\cJ_V$ and $\widetilde{M}_{(0,2H,-4)}(S)$ (cf.\ \cite[Proposition~4.4]{Sacca:BirationalGeometryLSV} for a more precise description of this family).
By Proposition~\ref{prop:b_V} and Proposition~\ref{prop:Theta}, 
there exists a parallel transport operator
\begin{equation*}
P_1\colon\oH^2(\cJ_V,\ZZ)\longrightarrow\oH^2(\widetilde{M}_{(0,2H,-4)}(S),\ZZ)
\end{equation*}
in the family $p_1$, such that $P_1(b_V)=b$ and $P_1(\Theta_V)=a-2b$.

Now, since $\widetilde{M}_{(0,2H,-4)}(S)$ is birational to $\widetilde{M}_S$, there is a family 
$p_2\colon\cX_2\to\widetilde{\Delta}$ over the disc with two origins 
(cf.\ \cite[Theorem~4.6]{Huybrechts:BasicResults}) such that the two origins correspond to 
$\widetilde{M}_{(0,2H,-4)}(S)$ and $\widetilde{M}_S$. 
We constructed in Sections \ref{subsection:hyperelliptic} and \ref{subsection:O'Grady} a parallel transport operator
\begin{equation}\label{eqn:P_2}
P_2\colon\oH^2(\widetilde{M}_{(0,2H,-4)}(S),\ZZ)\longrightarrow\oH^2(\widetilde{M}_S,\ZZ)
\end{equation}
in the family $p_2$, such that $P_2(a)=2H-3\widetilde{B}-2\widetilde{\Sigma}$, 
$P_2(b)=H-2\widetilde{B}-\widetilde{\Sigma}$ and
$P_2(\sigma)=3H-6\widetilde{B}-4\widetilde{\Sigma}$ (see Proposition~\ref{prop:psi} and Proposition~\ref{prop:G}).

Gluing together these two families, 
we eventually get a family $\cX\to T$ with two distinguished points corresponding to 
$\cJ_V$ and $\widetilde{M}_S$, and a parallel transport operator
\begin{equation}\label{eqn:P}
P=P_2\circ P_1\colon\oH^2(\cJ_V,\ZZ)\longrightarrow\oH^2(\widetilde{M}_S,\ZZ)
\end{equation}
such that 
$$P(\Theta_V)=\widetilde{B}\qquad\mbox{and}\qquad P(b_V)=H-2\widetilde{B}-\widetilde{\Sigma}=f.$$
Notice that 
$$e=H-\widetilde{B}-\widetilde{\Sigma}=P(\Theta_V+b_V),$$
therefore the theorem follows at once by Proposition~\ref{thm:my monodromy}.
\begin{rmk}
Let $M_{(0,2H,0)}(S)$ be the moduli space containing the Jacobian fibration $\cJ^4_{|2H|}$ of
degree $4$ divisors on the smooth curves in $|2H|$. 
By Lemma~\ref{lemma:0=4}, $M_{(0,2H,0)}(S)$ is isomorphic
to $M_{(0,2H,-4)}(S)$ and the image of $T_S$ under
this isomorphism is the (closure of the) theta divisor $E_S$ of effective line bundles in
$\cJ^4_{|2H|}$. The latter divisor is birational to the symmetric product $\Sym^4\cC$, where
$\cC$ is the universal family of (smooth) curves in $|2H|$. There is a natural map 
$\Sym^4\cC\to \Sym^4 S$, whose generic fibre is the $\PP^1$ of curves in $|2H|$ passing 
through four fixed points. 

On the other hand, by \cite[Proposition~3.0.5]{O'Grady:10dimPublished}, there is a morphism
$\widetilde{B}\to\Sym^4S$ whose generic fibre is again $\PP^1$.

The parallel transport operator $P_2$ makes rigorous the natural expectation that $T_S$ deforms to $\widetilde{B}$.
\end{rmk}

%%%%%%%%%%%%%%%%%%%%%%%%%%%%%%%%%%%%%%%%%
%%%%%%%%%%%%%%%%%%%%%%%%%%%%%%%%%%%%%%%%%
\section{The monodromy group}\label{section:Mon^2}
Let $S$ be a projective $K3$ surface such that $\Pic(S)=\ZZ H$ with $H^2=2$.
Following the notation introduced in the previous section, we put
$e=H-\widetilde{B}-\widetilde{\Sigma}$ and $f=H-2\widetilde{B}-\widetilde{\Sigma}$, 
and denote by $\overline{U}$ the hyperbolic plane they generate. 
Notice that $\overline{U}=P_2((v^\perp)^{1,1})$, where $v=(0,2H,-4)$ and $P_2$ is the parallel 
transport operator (\ref{eqn:P_2}).

Let $A$ be the projection of $\widetilde{\Sigma}$ on the orthogonal complement $\overline{U}^\perp$ of $\overline{U}$, that is $A$ is such that
$\widetilde{\Sigma}=3f-A$ and $A\perp \overline{U}$.
\begin{lemma}\label{lemma:A monodromy}
The reflection $R_A$ is a monodromy operator.
\end{lemma}
\begin{proof}
Let $\sigma$ be the class of the exceptional divisor $\widetilde{\Sigma}_v$, where 
$v=(0,2H,-4)$, and let $P_2$ be the parallel transport operator defined in (\ref{eqn:P_2}). 
Since $P_2(\sigma)=A$, by \cite[Proposition~5.4]{Markman:PrimeExceptional} $A$ is a stably prime exceptional divisors, and so the reflection $R_A$ is a monodromy operator by \cite[Theorem~1.1]{Markman:PrimeExceptional}.
\end{proof}
\begin{rmk}\label{rmk:A acts -id}
Notice that $R_A$ is the identity on $\overline{U}$ and acts as $-id$ on the discriminant group.
It follows from Proposition~\ref{thm:my monodromy} that $R_A$ is not induced by a family of cubic fourfolds via the LSV construction.
\end{rmk}

In order to keep the notation as easy as possible, from now on we simply denote by 
$\Or^+$ the group $\Or^+(\oH^2(\widetilde{M}_S,\ZZ))$. 

Consider the following groups:
\begin{eqnarray}
G_1 & = & \left\{g\in\Or^+\mid g(H)=H,\; g(\widetilde{B})=\widetilde{B},\; g(\widetilde{\Sigma})=\widetilde{\Sigma}\right\} \\ \nonumber
G_2 & = & \left\{g\in\Or^+\mid g(\xi)=\xi, \;\forall\xi\in H^2(S,\ZZ)\right\}=\langle R_{\widetilde{B}},R_{\widetilde{\Sigma}}\rangle\\ \nonumber
G_3 & = & \left\{g\in\widetilde{\Or}^+\mid g(e)=e,\; g(f)=f\right\}. \nonumber
\end{eqnarray}
Let $k\in\oH^2(S,\ZZ)$ be a class such that $k^2=-2$ and $(k,H)=0$, and put
$l=k+f$. Notice that $R_k\in G_1$.

Define
$$G=\langle G_1,G_2,G_3,R_l,R_A\rangle.$$
\begin{prop}\label{prop:Mon}
$G\subset\Mon^2(\widetilde{M}_S)$.
\end{prop}
\begin{proof}
First of all let us notice that $G_1,G_3\subset\Mon^2(\widetilde{M}_S)$ by 
Theorem~\ref{thm:my operators O'Grady} and Theorem~\ref{thm:LSV to OG}, and that $G_2\subset\Mon^2(\widetilde{M}_S)$ by \cite[Section~5.2]{Markman:ModularGaloisCovers}. 
Moreover, if we specialise $S$ to a
$K3$ surface $S_0$ as in Example~\ref{example:R_L}, and we choose a parallel transport
operator $P\colon\oH^2(S,\ZZ)\to\oH^2(S_0,\ZZ)$ such that $P(H)=H$ and $P(k)=K$, then 
$R_l=P^{-1}\circ R_{\widetilde{Z}}\circ P\in\Mon^2(\widetilde{M}_S)$, where $R_{\widetilde{Z}}$ 
is the reflection in Example~\ref{example:R_L}.  
Finally, $R_{A}\in\Mon^2(\widetilde{M}_S)$ by Lemma~\ref{lemma:A monodromy}.
\end{proof}

\begin{thm}\label{thm:Mon^2}
$\Mon^2(\widetilde{M}_S)=G=\Or^+(\oH^2(\widetilde{M}_S,\ZZ))$.
\end{thm}
The proof of the theorem is lattice-theoretic, so we recall here the notation and
the results we need. Let $L$ be an even lattice. If $z\in L$ is an isotropic element, i.e.\ $z^2=0$, 
and $a\in L$ is orthogonal to $z$, then the transvection $t(z,a)$ is defined by
$$ t(z,a)(x)=x-(a,x)z+(z,x)a-\frac{1}{2}(a,a)(z,x)z.$$
Transvections are orientation preserving isometries with determinant $1$ and acting as the 
identity on the discriminant group.
\begin{lemma}[\protect{\cite[Section~3]{GritsenkoHulekSankaran:Abelianisation}}]\label{lemma:transvections}
\begin{enumerate}
\item $t(z,a)^{-1}=t(z,-a)=t(-z,a)$;
\item $g\circ t(z,a)\circ g^{-1}=t(g(z),g(a))$ for every $g\in\Or^+$;
\item if $R_a$ is integral, then $t(z,a)=R_aR_{a+\frac{1}{2}a^2z}$;
\item if $(a,z)=0=(b,z)$, then $t(z,a+b)=t(z,a)\circ t(z,b)$.
\end{enumerate}
\end{lemma}
Now suppose that $L=U\oplus L_1$ and that the hyperbolic plane $U$ is generated by two isotropic classes $e$ and $f$; define 
$$E_U(L_1)=\langle t(e,a),t(f,a)\mid a\in L_1\rangle.$$
If moreover $L_1=U\oplus L_2$, then by \cite[Proposition~3.3~(iii)]{GritsenkoHulekSankaran:Abelianisation})
\begin{equation}\label{eqn:Eichler}
\Or^+(L)=\langle E_U(L_1),\Or^+(L_1)\rangle,
\end{equation}
where $\Or^+(L_1)$ is embedded in $\Or^+(L)$ by extending any isometry of $L_1$ via the 
identity.

We are now ready to prove Theorem~\ref{thm:Mon^2}.
\begin{proof}[Proof of Theorem~\ref{thm:Mon^2}]
By Remark~\ref{rmk:orient pub} and Proposition~\ref{prop:Mon} we have a chain of inclusions
$$G\subset\Mon^2(\widetilde{M}_S)\subset\Or^+(\oH^2(\widetilde{M}_S,\ZZ)).$$

We claim that $G=\Or^+(\oH^2(\widetilde{M}_S,\ZZ))$, from which the theorem follows.

Recall that $\overline{U}\subset\Or^+(\oH^2(\widetilde{M}_S,\ZZ))$ is the hyperbolic plane spanned by the distinguished classes $e=H-\widetilde{B}-\widetilde{\Sigma}$ and $f=H-2\widetilde{B}-\widetilde{\Sigma}$; denote by $L_1$ its orthogonal complement.
Then by the identification (\ref{eqn:Eichler}), 
$$\Or^+(\oH^2(\widetilde{M}_S,\ZZ))=\langle E_{\overline{U}}(L_1),\Or^+(L_1)\rangle.$$
Notice that $\Or^+(L_1)=\langle G_3,R_A\rangle$. In fact $G_3\cong\widetilde{\Or}^+(L_1)$ by definition, and $R_A$ restricts to an isometry of $L_1$ acting as $-\id$ on the discriminant group (Remark~\ref{rmk:A acts -id}).

It follows that it is enough to show that all the transvections $t(e,a)$ and $t(f,a)$, with $a\in L_1$, belong to 
$G$.

By part (2) of Lemma~\ref{lemma:transvections}, one notices that 
$$R_{\widetilde{B}}\circ t(f,a)\circ R_{\widetilde{B}}=t(R_{\widetilde{B}}(f),R_{\widetilde{B}}(a))=t(e,a),$$ where we used that $R_{\widetilde{B}}(a)=a$, for every $a\in L_1$, since $\widetilde{B}\in\overline{U}$.
So it is enough to prove the claim for $t(f,a)$.

Choosing a base $\{a_1,\cdots,a_{22}\}$ for 
$L_1\cong U^2\oplus E_8(-1)^2\oplus A_2(-1)$, by part (4) of Lemma~\ref{lemma:transvections} it is enough to prove the claim
for $t(f,a_1),\cdots,t(f,a_{22})$. Notice that there is a canonical basis for $L_1$ with 
$(a_i,a_i)=0$ or $(a_i,a_i)=-2$: in both cases
$a_i$ has divisibility $1$.

On the other hand, for any isotropic element $c\in L_1$, there exist two $(-2)$-elements $a,b\in L_1$ such that $t(f,c)=t(f,a)\circ t(f,b)$. 
In fact, if $a\in L_1$ is a $(-2)$-element such that $(a,c)=0$, then pick $b=c-a$.

In this way, we reduce the problem to proving that $t(f,a)\in G$ for every 
$(-2)$-element $a\in L_1$. Applying the Eichler criterion 
\cite[Proposition~3.3~(i)]{GritsenkoHulekSankaran:Abelianisation} to the lattice $L_1$, and using
part (2) of Lemma~\ref{lemma:transvections}, we eventually notice that it is enough to 
prove the claim for one specific $a$.

Let $a=k\in\oH^2(S,\ZZ)$ be the class in $\oH^2(S,\ZZ)$ such that $k^2=-2$ and $(H,k)=0$.
Since the reflection $R_k$ is integral, by part (3) of Lemma~\ref{lemma:transvections} 
we can write
$$t(-f,k)=R_k R_{k+f},$$
and $t(-f,k)=t(f,k)^{-1}$ by part (1) of Lemma~\ref{lemma:transvections}.
Finally, $R_k\in G_1$ and $R_{k+f}=R_l\in G$, so the claim is proved.
\end{proof}

\begin{rmk}
The proof proves the stronger statement
$$ \Or^+(\oH^2(\widetilde{M}_S,\ZZ))=\langle R_k, R_{\widetilde{B}}, R_l, R_A, G_3\rangle.$$
\end{rmk}

%%%%%%%%%%%%%%%%%%%%%%%%%%%%%%%%%%%%%%%%%
%%%%%%%%%%%%%%%%%%%%%%%%%%%%%%%%%%%%%%%%%
%%%%%%%%%%%%%%%%%%%%%%%%%%%%%%%%%%%%%%%%%
\section{The locally trivial monodromy group of the singular moduli space}
In this section we explain how Theorem~\ref{thm:Mon^2} helps to compute the locally trivial monodromy group of the singular moduli spaces $M_S$. We refer to Section~\ref{section:Singular} for the definitions.

Let us recall that there exists a symplectic resolution of singularities $\pi\colon\widetilde{M}_S\to M_S$. The monodromy group $\Mon^2(\widetilde{M}_S)$ and the locally trivial monodromy group $\Mon^2(M_S)_{\operatorname{lt}}$ are related by means of the monodromy group $\Mon^2(\pi)$ of simultaneous monodromy operators. Recall that $\Mon^2(\pi)\subset\Or(\oH^2(\widetilde{M}_S,\ZZ))\times\Or(\oH^2(M_S,\ZZ))$, and denote by $p$ and $q$ the two projections, i.e.
\begin{equation*}
\xymatrix{
                                             &   \Mon^2(\pi)\ar[dl]_{p}\ar[dr]^{q}  &   \\
\Or(\oH^2(\widetilde{M}_S,\ZZ))  &                                                & \Or(\oH^2(M_S,\ZZ)).
}
\end{equation*}

\begin{thm}\label{thm:l t monodromy}
$\Mon^2(M_S)_{\operatorname{lt}}=\Ort^+(\oH^2(M_S,\ZZ))$.
\end{thm}
\begin{proof}
First of all, the projection $p\colon\Mon^2(\pi)\to\Or(\oH^2(\widetilde{M}_S,\ZZ))$ is injective and moreover, by the item $(1)$ of \cite[Corollary~5.18]{BakkerLehn:Singular2016}, $\Mon^2(\pi)$ is identified with the subgroup of $\Mon^2(\widetilde{M}_S)$ stabilising the resolution K\"ahler chamber (see \cite[Definition~5.4]{BakkerLehn:Singular2016}). In this case the resolution K\"ahler chamber is the chamber containing the K\"ahler cone in the decomposition $$\cC_{\widetilde{M}_S}\setminus\widetilde{\Sigma}^\perp.$$ Here $\cC_{\widetilde{M}_S}$ is the positive cone, i.e.\ the connected component containing the K\"ahler cone in the cone of positive classes in $\oH^{1,1}(\widetilde{M}_S,\ZZ)$.

Since $\Mon^2(\widetilde{M}_S)=\Or^+(\oH^2(\widetilde{M}_S,\ZZ))$ by Theorem~\ref{thm:Mon^2}, it follows that 
$$ \Mon^2(\pi)=\Or^+(\oH^2(\widetilde{M}_S,\ZZ),\widetilde{\Sigma}),$$
namely the subgroup of isometries $g$ such that $g(\widetilde{\Sigma})=\widetilde{\Sigma}$.

Finally, by item $(2)$ of \cite[Corollary~5.18]{BakkerLehn:Singular2016}, the locally trivial monodromy group $\Mon^2(M_S)_{\operatorname{lt}}$ is identified with the image of the projection $q\colon\Mon^2(\pi)\longrightarrow\Or(\oH^2(M_S,\ZZ)).$

By \cite[Proposition~1.5.1]{Nikulin:Lattice}, the image of $q$ is identified with the subgroup of isometries $h\in\Or(\oH^2(M_S,\ZZ))$ such that $h$ acts as the identity on the finite group $\oH^2(\widetilde{M}_S,\ZZ)/(\pi^*\oH^2(M_S,\ZZ)\oplus\ZZ\widetilde{\Sigma})$. Since the last group is isomorphic to the discriminant group of $\oH^2(M_S,\ZZ)$, the theorem is proved.
\end{proof}

\begin{rmk}
Geometrically Theorem~\ref{thm:l t monodromy} reflects the fact that there are two singular moduli spaces that are birational, but whose singular locus is not preserved under the birational isomorphism (cf.\ Section~\ref{subsection:hyperelliptic}). More precisely, the birational isomorphism does not preserve the singularity type of the two moduli spaces: one has locally factorial singularities, while the other has $2$-factorial singularities.
\end{rmk}

%%%%%%%%%%%%%%%% BIBLIOGRAPHY %%%%%%%%%%%%%%%
\bibliographystyle{alpha}
% here we change the meaning of \VAN to use the prefix for the bibliography
\DeclareRobustCommand{\VAN}[3]{#3}
\bibliography{Bibliography}
%\addcontentsline{toc}{chapter}{Bibliography}

\end{document}